\documentclass[11pt, reqno]{amsart}

\usepackage{amsmath}

\numberwithin{equation}{section}

\usepackage[colorlinks=true,pagebackref,hyperindex]{hyperref}
\usepackage{mathrsfs}

\usepackage[all]{xy}
\usepackage{color}
\usepackage{amsfonts}
\usepackage{amscd}
\usepackage{amssymb}

\definecolor{darkgreen}{rgb}{0.0, 0.7, 0.0}

\definecolor{cyan}{cmyk}{1,0,0,0}

\newtheorem{tm}{Theorem}[section]
\newtheorem{lm}[tm]{Lemma}
\newtheorem{pr}[tm]{Proposition}
\newtheorem{rmk}[tm]{Remark}
\newtheorem{cor}[tm]{Corollary}

\newtheorem{??}[tm]{Question}

\newtheorem*{thmA}{Theorem A}
\newtheorem*{thmB}{Theorem B}
\newtheorem*{corC}{Corollary C}
\newtheorem*{thmD}{Theorem D}
\newtheorem*{thmE}{Theorem E}

\newtheorem*{thmF}{Theorem F}

\newcommand{\ben}{\begin{enumerate}}
\newcommand{\een}{\end{enumerate}}
\newcommand{\bit}{\begin{itemize}}
\newcommand{\eit}{\end{itemize}}
\newcommand{\beq}{\begin{equation}}
\newcommand{\eeq}{\end{equation}}
\newcommand{\la}{\label}

\newcommand\ci{\cite}

\font\tenmsb=msbm10
\font\sevenmsb=msbm7
\font\fivemsb=msbm5
\newfam\msbfam
\textfont\msbfam=\tenmsb
\scriptfont\msbfam=\sevenmsb
\scriptscriptfont\msbfam=\fivemsb
\def\Bbb#1{{\fam\msbfam #1}}
\font\teneufm=eufm10
\font\seveneufm=eufm7
\font\fiveeufm=eufm5
\newfam\eufmfam
\textfont\eufmfam=\teneufm
\scriptfont\eufmfam=\seveneufm
\scriptscriptfont\eufmfam=\fiveeufm

\newcommand\rat{{\Bbb Q}}

\newcommand\comp{{\Bbb C}}
\newcommand\real{{\Bbb R}}
\newcommand\zed{{\Bbb Z}}

\newcommand\s{\sigma}

\begin{document}

\title{Combinatorics and  topology of proper toric maps}

\author[M.A.~de Cataldo]{Mark Andrea A.  de Cataldo}
\address{Department of Mathematics, Stony Brook University,
Stony Brook, NY 11794-3651, USA}
\email{{\tt  mark.decataldo@stonybrook.edu}}

\author[L.~Migliorini]{Luca~Migliorini}
\address{Dipartimento di Matematica, Universita di Bologna,
Piazza di Porta S. Donato 5, 40127 Bologna, Italy}
\email{{\tt luca.migliorini@unibo.it}}

\author[M. Musta\c{t}\u{a}]{Mircea~Musta\c{t}\u{a}}
\address{Department of Mathematics, University of Michigan,
Ann Arbor, MI 48109, USA}
\email{{\tt mmustata@umich.edu}}

\markboth{M.~de Cataldo, L.~Migliorini, and M.~Musta\c t\u a}{The combinatorics and topology of proper toric maps}

\thanks{The research of de Cataldo was partially supported by NSF
grant DMS-1301761 and by a grant from the Simons Foundation (\#296737 to Mark de Cataldo).
The research of Migliorini was partially supported by PRIN project 2012 ``Spazi di moduli e teoria di Lie''.
The research of Musta\c{t}\u{a} was partially supported by NSF grant DMS-1401227.}

\begin{abstract}
We study the topology of toric maps. We show that if $f\colon X\to Y$ is a proper toric morphism, with $X$ simplicial, then
the cohomology of every fiber of $f$ is pure and of Hodge-Tate type. When the map is a fibration, we give an explicit formula for the Betti numbers
of the fibers in terms of a relative version of the $f$-vector, 
extending the usual formula for the Betti numbers of a simplicial complete toric variety. We then describe the Decomposition Theorem
for a toric fibration, giving in particular a nonnegative combinatorial invariant attached to each cone in the fan of $Y$, which is positive
precisely when the corresponding closed subset of $Y$ appears  as a support in the Decomposition Theorem. The description of this invariant involves 
the stalks of the intersection cohomology complexes on $X$ and $Y$, but in the case when both $X$ and $Y$ are simplicial, there is a simple formula
in terms of the relative $f$-vector.
\end{abstract}

\maketitle


\section{Introduction}
A complex toric variety is a normal complex algebraic variety $X$ that carries an action of a torus $T=({\mathbf C}^*)^n$, such that $X$ has an open orbit isomorphic 
to $T$. Toric varieties can be described in terms of convex-geometric objects, namely fans and polytopes, and part of their appeal comes from the fact
that algebro-geometric properties of the varieties translate into combinatorial properties of the fans or polytopes, see \cite{Ful93}. For example, 
given a complete simplicial toric variety $X$ (recall that \emph{simplicial} translates as having quotient singularities), the information given by the Betti numbers of $X$
is equivalent to that encoded by the $f$-vector of $X$, which records the number of cones of each dimension in the fan defining $X$. This has been famously used by
Stanley in \cite{Stanley1}, together with the Poincar\'{e}  duality and Hard Lefschetz theorems to prove a conjecture of McMullen concerning the $f$-vector of a simple polytope. 
When $X$ is not-necessarily-simplicial, it turns out that the right cohomological invariant to consider is not the cohomology of $X$ (which is \emph{not} a combinatorial invariant),
but the intersection cohomology (see \cite{Stanley2} and also \cite{DL} and \cite{Fieseler}).

In this paper we are concerned with a relative version of this story. More precisely, we consider a proper (equivariant) morphism of toric varieties $f\colon X\to Y$ and study two related questions. 
We first study the cohomology of the fibers of $f$ and then apply this study  to describe the Decomposition Theorem for $f.$ Our results are most precise when $f$ is a fibration, that is,
when it is proper, surjective and with connected fibers.

We begin with the following result about simplicial toric varieties that admit a proper toric map to an affine toric variety that has a fixed point. The case of complete simplicial
toric varieties is well-known. Recall that a pure Hodge structure of weight $q$ is \emph{of Hodge-Tate type} if all Hodge numbers $h^{i,j}$ are $0$,
unless $i=j$ (in particular, the underlying vector space is $0$ if $q$ is odd).

\begin{thmA}
If $X$ is a simplicial toric variety that admits a proper toric map to an affine toric variety that has a fixed point, then the following hold for every $q$:
\begin{enumerate}
\item[i)] The canonical map $A_q(X)_{\rat}\to H_{2q}^{BM}(X)_{\rat}$ is an isomorphism, where $A_q(X)_{\rat}$ is the Chow group of $q$-dimensional cycle classes and $H^{BM}_{2q}(X)_{\rat}$ is the $(2q)^{\rm th}$ Borel-Moore homology
group of $X$ (both with $\rat$-coefficients).
\item[ii)] The mixed Hodge structures on each of $H^q_c(X,\rat)$ and $H^q(X,\rat)$ are pure, of weight $q$, and of Hodge-Tate type.
\end{enumerate}
\end{thmA}

As a consequence, we deduce the following:

\begin{thmB}
Let $f\colon X\to Y$ be a proper toric map between complex toric varieties, with  $X$  simplicial. For every $y\in Y$ and every $q$,
the mixed
Hodge structure on $H^q(f^{-1}(y),{\mathbf Q})$ is pure, of weight $q$, and of Hodge-Tate type.
\end{thmB}

By virtue of Proposition \ref{str_fibers} below,  under the assumptions of  Theorem~B, every irreducible component of $f^{-1}(y)$ is a complete, simplicial toric variety, for which the properties
in the statement are well-known. However, the fact that the union still satisfies these properties is not a general fact. In order to prove Theorem~B,
we first  reduce to the case when $X$ is smooth, $f$ is a projective fibration, $Y$ is affine, and has a torus-fixed point  which is equal to $y$. In this case,
there is an isomorphism of mixed Hodge structures
$$H^q(f^{-1}(y),\rat)\simeq H^q(X,\rat)$$ and therefore Theorem~A implies Theorem~B. In order to prove Theorem~A, we prove Corollary \ref{cor_filtration},
which yields  a filtration of $X$  by open subsets $\emptyset=U_0\subset U_1\subset\ldots\subset U_r=X$
such that each difference $U_i\smallsetminus U_{i-1}$ is isomorphic to a quotient of an affine space by a finite group. The existence of such a filtration is well-known when $X$
is smooth and projective (see \cite[Chapter~5.2]{Ful93}) and a similar augment gives it in the context we need. It is then easy to deduce the assertions
about the cohomology of $X$ from the existence of such a filtration.  We remark  that the purity statements
in both theorems can also be deduced from the above isomorphism of mixed Hodge
structures via some general properties of the weight filtration (see Remark~\ref{weights_trick}). However, we have  preferred to give the argument described above 
in order to emphasize the elementary nature of the results.

For every toric fibration $f\colon X\to Y$, it is easy to compute the Hodge-Deligne polynomial of the fibers of $f$. This is done in terms of a relative version
of the familiar notion of $f$-vector of a toric variety. Recall that the orbits of the torus action on a toric variety $Y$ are in bijection with the cones in the fan
$\Delta_Y$ defining $Y$, with the orbit $O(\tau)$ corresponding to $\tau\in\Delta_Y$ having dimension equal to ${\rm codim}(\tau)$. Moreover, the irreducible
torus-invariant closed subsets of $Y$ are precisely the orbit closures $V(\tau)=\overline{O(\tau)}$. With this notation, for every cone $\tau\in\Delta_Y$,
we denote by 
$d_{\ell}(X/\tau)$ the number of irreducible torus-invariant closed subsets $V(\sigma)$ of $X$ such that $f(V(\sigma))=V(\tau)$ and 
$\dim(V(\sigma))-\dim(V(\tau))=\ell$. The following corollary (Corollary \ref{Betti_fib}) generalizes the well-known formula
for the Betti numbers of a complete, simplicial toric variety. Note that  Theorem~B
implies that in what follows, the odd cohomology groups are trivial.

\begin{corC}
If $f\colon X\to Y$ is a toric fibration, with $X$ simplicial, and $y\in O(\tau)$ for some cone $\tau\in \Delta_Y$, then
for every $m\in{\mathbf Z}_{\geq 0}$, we have 
$$\dim_{\mathbf Q}H^{2m}(f^{-1}(y),{\mathbf Q})=\sum_{\ell\geq m}(-1)^{\ell-m}{{\ell}\choose m}d_{\ell}(X/\tau).$$
Moreover, the Euler-Poincar\'{e} characteristic of $f^{-1}(y)$ is given by $\chi(f^{-1}(y))=d_0(X/\tau)$.
\end{corC}

Our  initial goal  was to understand the decomposition theorem of \cite{BBD} for
proper  toric morphisms. In the setting of finite fields, this is carried out in \ci{deC2}. In our setting, i.e. over the complex numbers,
the theorem takes the following form
(for a variety $X$, we denote by $IC_X$ the intersection complex on $X;$
for nonsingular $X,$ this is 
${\mathbf Q}_X [\dim{X}]$).

\begin{thmD}
If $X$ and $Y$ are complex toric varieties and $f\colon X\to Y$ is a toric fibration, then we have a decomposition
\begin{equation}\label{eq_toric_dec_thm_intro}
Rf_*(IC_X)\simeq \bigoplus_{\tau\in\Delta_Y}\bigoplus_{b\in\zed}IC_{V(\tau)}^{\oplus s_{\tau,b}}[-b],
\end{equation} 
where the nonnegative integers $s_{\tau,b}$ satisfy $s_{\tau,b}=0$ if $b+\dim(X)-\dim(V(\tau))$ is odd.
\end{thmD}

In fact, the integers $s_{\tau,b}$ in the above theorem satisfy further constraints coming from Poincar\'{e} duality and the relative Hard Lefschetz theorem;
see Theorem~\ref{toric_dec_thm} below for the precise statement.
By comparison with the general statement of the decomposition theorem from \cite{BBD}, there are two points. Firstly, the subvarieties that appear in the
decomposition (\ref{eq_toric_dec_thm_intro}) are torus-invariant and, secondly, the intersection complexes that appear  have constant coefficients
(this is due to the fact that the map is a fibration). 

We are interested in the \emph{supports} of toric fibrations, that is, the subvarieties $V(\tau)$ that appear in the decomposition
(\ref{eq_toric_dec_thm_intro}). More precisely, in the setting of Theorem~D, for every $\tau\in\Delta_Y$ we define the key invariant of this paper
$\delta_{\tau}:=\sum_{b\in\zed}s_{\tau,b}$. Clearly, in view of (\ref{eq_toric_dec_thm_intro}), we have that the closure of an orbit $V(\tau)$ is a support if and only if $\delta_{\tau}>0$. Theorems E and F  below relate the topological  invariant $\delta$
of the toric fibration $f$
to the associated combinatorial data.

Let us start by discussing the simpler case of Theorem~E, where both varieties are simplicial.
By building on Corollary~C, we obtain the following description.

\begin{thmE}
If $f\colon X\to Y$ is a toric fibration, where $X$ and $Y$ are simplicial toric varieties, then for every cone $\tau\in\Delta_Y$, we have
\begin{equation}\label{eq_thmD}
\delta_{\tau}=\sum_{\sigma\subseteq\tau}(-1)^{\dim(\tau)-\dim(\sigma)}d_0(X/\sigma).
\end{equation}
\end{thmE}

In fact, in this simplical case, we obtain an explicit formula (see Theorem~\ref{form_both_simplicial}) for each of the numbers $s_{\tau,b}$. An interesting consequence of Theorem~E is that the expression on the right-hand side of (\ref{eq_thmD}) is 
nonnegative.
It would be desirable to find a direct combinatorial argument for this fact. When $f$ is birational and 
$\dim(\tau)\leq 3$, we give a combinatorial description of $\delta_{\tau}$ which implies that it is nonnegative (see Remark~\ref{rmk_special_case}).
However,  we do not know of  such a formula when $\tau$ has higher dimension.

It is worth noting that the explicit formula for the invariants $s_{\tau,b}$ in the simplicial case, in combination with Poincar\'{e} duality and Hard Lefschetz, 
leads to interesting constraints on the combinatorics of the morphism (see Remark~\ref{rmk_rel_f_vector} for the precise statement). These constraints extend to the relative setting
the well-known conditions on the $f$-vector of a simplicial, projective toric variety.

In order to treat the case of fibrations between not-necessarily-simplicial toric varieties, we need to introduce some notation. Given a toric variety $Y$ and two cones
$\tau\subseteq\sigma$ in the fan $\Delta_Y$ defining $Y$, we put $r_{\tau,\sigma}:=\dim_{\rat}{\mathcal H}^*(IC_{V(\tau)})_{x_{\sigma}}$, where $x_{\sigma}$ can be taken to be 
any point in the orbit $O(\sigma)\subseteq V(\tau)$. It is a consequence of the results in \cite{Fieseler}  and \cite{DL}, independently,  that $r_{\tau,\sigma}$ is a combinatorial invariant.
In turn, we have the invariant $\widetilde{r}_{\tau,\sigma}$ for cones $\tau\subseteq\sigma$ in $\Delta_Y$, uniquely determined by the property that for every $\tau\subseteq\sigma$, the sum
$\sum_{\tau\subseteq\gamma\subseteq\sigma}r_{\tau,\gamma}\cdot \widetilde{r}_{\gamma,\sigma}$ is equal to $1$ if $\tau=\sigma$ and it is equal to $0$, otherwise.
By \cite{Stanley4},  $\widetilde{r}_{\tau,\sigma}$ coincides, up to sign, with the $r_{\tau, \sigma}$-function of the dual poset. 
Suppose now that $f\colon X\to Y$ is a toric fibration. For a cone $\sigma\in\Delta_Y$, we put $p_{f,\sigma}:=\dim_{\rat}{\mathcal H}^*(f^{-1}(x_{\sigma}),IC_X)$, where again
we may take $x_{\sigma}$ to be any point in the orbit $O(\sigma)$. The next result gives a description of $\delta_{\sigma}$ in terms of the above invariants. The second part
implies that the invariants $p_{f,\sigma}$, hence also the $\delta_{\sigma}$, are combinatorial.

\begin{thmF}
With the above notation, if $f\colon X\to Y$ is a toric fibration, then the following hold:
\begin{enumerate}
\item[i)] For every cone $\tau\in\Delta_Y$, we have $\delta_{\tau}=\sum_{\sigma\subseteq\tau}\widetilde{r}_{\sigma,\tau}\cdot p_{f,\sigma}$.
\item[ii)] For every cone $\sigma\in\Delta_Y$, we have $p_{f,\sigma}=\sum_ir_{0,\sigma_i}$, where the $\sigma_i$ are the cones in the fan $\Delta_X$ defining $X$
with the property that $f(V(\sigma_i))=V(\sigma)$ and $\dim(V(\sigma_i))=\dim(V(\sigma))$.
\end{enumerate}
\end{thmF}

Just as for Theorem E, an explicit formula for each of the numbers $s_{\tau,b}$ is obtained by upgrading $r_{\sigma, \tau}$ and $p_{f,\sigma}$ to Laurent polynomials with integral coefficients in order to keep track of the cohomological grading (see Theorem \ref{form_general} for the precise result).
Again, as in Theorem~E, note that while $\delta \geq 0$ by definition, the right-hand side 
of Theorem~F.i) contains the factors $\widetilde{r}$, which have ``alternating signs",
so that, in particular, we find this right-hand side to be nonnegative as well. In this paper, this is proved as a consequence of the decomposition theorem. As in the simplicial case, it would be desirable to find a combinatorial description for $\delta$ that implies its non-negativity. 

We mention that in their recent preprint \cite{KS}, Katz and Stapledon undertake a related study in a more combinatorial framework,
focused on invariants associated to certain maps between posets. In the case of a toric fibration $f\colon X\to Y$, their
results apply to the map $f_*\colon\Delta_X\to\Delta_Y$. In particular, our invariants $s_{\tau,b}$ appear in their setting as the coefficients
of a \emph{local h-polynomial}. For some comparisons between their results and ours, see Remarks~\ref{rmk_KS1}, \ref{rmk_KS2}, and \ref{rmk_KS3}.


\bigskip

The paper is organized as follows. In Section 2 we review the basics of toric geometry that we use. 
We pay special attention to some facts concerning toric morphisms that seem somewhat less known, such as the description
of toric fibrations and Stein factorizations in the toric setting, and the description of the irreducible components of
fibers of toric maps. In Section 3 
we study the cohomology of toric varieties that admit proper toric maps to affine toric varieties that have a fixed point. In particular,
we prove Theorem~A (see Theorems~\ref{BM} and 
\ref{pure_coh}).
In Section 4, we apply the results in the previous section to obtain Theorem~B (see Theorem~\ref{pure_coh}). We use this and the computation
of Hodge-Deligne polynomials for fibers of toric fibrations to give the formula for the Betti numbers of such fibers in Corollary~C
(see Corollary~\ref{Betti_fib}).
In Section 5 we deduce Theorem~D  
from the decomposition theorem in \cite{BBD} (see Theorem~\ref{toric_dec_thm}). Section 6 is devoted to the description of the invariants $\delta_{\sigma}$ in the case
of a toric fibration between simplicial toric varieties (see Theorem~\ref{form_both_simplicial}), while in Section~7 we treat the general case 
(see Theorems~\ref{form_general} and \ref{thm_p_sigma}).

\subsection*{Acknowledgments}
The first-named author is grateful to the Max Planck Institute of Mathematics in Bonn
for the perfect working conditions.
We are grateful to Tom Braden, William Fulton and Vivek Shende for helpful discussions. 
Laurentiu Maxim informed us that related results are contained in \cite{CMS}, especially Thm. 3.2. 
During the preparation of this paper we learned that E. Katz and A. Stapledon are investigating, although
from a rather different viewpoint,  similar questions. We thank them for informing us about their results 
and sending us a draft of their paper \cite{KS}. Last but not least, we are grateful to the anonymous referees
for their comments and suggestions.

\section{Basic toric algebraic geometry}\la{sec_basic}

In this section we review some basic facts about toric varieties and toric maps.
For all assertions that we do not prove in this section, as well as for all the standard notation for toric varieties that we employ,
we refer to
\cite{Ful93}.  We work over an algebraically closed field $k$, of arbitrary characteristic,
in the hope that some of the facts that we prove might be useful somewhere else, in this more general setting.

\subsection*{Toric varieties and their orbits}
A toric variety $X$  is associated with a lattice $N$ and a fan $\Delta$ in $N_{\real}=N\otimes_{\mathbf Z}\real$.
If $M$ is the dual lattice of $N$, then the torus $T_N:={\rm Spec}\,k[M]$
embeds as an open subset in $X$ and its standard action on itself extends to an action on $X$. 
We often write 
$N_X$, $M_X$, $\Delta_X$, and $T_X$ for the objects corresponding to a fixed toric variety $X$.
For every cone $\sigma\in\Delta$, there is an affine open subset $U_{\sigma}$ of $X$, with $U_{\sigma}={\rm Spec}\,k[\sigma^{\vee}\cap M]$.
These affine open subsets cover $X$.
The support of a fan $\Delta$ is the subset $|\Delta|=\bigcup_{\sigma\in \Delta}\sigma$ of $N_{\real}$. The toric variety $X$ is complete if and only if
$|\Delta_X|=N_{\real}$.

The orbits of the $T_X$-action on $X$ are in bijection with the cones in $\Delta_X$. The orbit $O({\sigma}):={\rm Spec}\,k[M_X\cap\sigma^{\perp}]$ corresponding
to $\sigma\in\Delta_X$ is a torus of dimension equal to ${\rm codim}(\sigma)$. 
The distinguished element of $O({\sigma})$
(the identity of the group) is denoted by $x_{\sigma}$. In particular, the smallest cone $\{0\}$ corresponds to the open orbit $T_X$.

The irreducible torus-invariant closed subsets of $X$ are precisely the orbit closures $V(\sigma):=\overline{O({\sigma})}$. 
The lattice corresponding to $V(\sigma)$ is $N/N_{\sigma}$, where $N_{\sigma}$ is the intersection of $N$ with the linear span of $\sigma$.
We always view $\Delta$ as a poset, ordered by the inclusion of cones.
Note that $\sigma\subseteq\tau$ if and only if $V(\sigma)\supseteq V({\tau})$. 
The open subset $U_{\sigma}$ is the union of those $O(\tau)$ with $\tau\subseteq\sigma$. 

Each $V(\sigma)$ is a toric variety, with corresponding torus $O({\sigma})$. There is a surjective morphism of algebraic groups
$T_X\to O({\sigma})$ such that the $T_X$-action on $V(\sigma)$ induces the $O({\sigma})$-action.
This morphism corresponds to the split inclusion $M_X\cap\sigma^{\perp}\hookrightarrow M_X$.
If $X$ is smooth or simplicial, then each $V(\sigma)$ has the same property. 

We say that an affine toric variety $U_{\sigma}$ is \emph{of contractible type} if $\sigma$ is a cone that spans $N_{\real}$.
Note that in this case $x_{\sigma}\in U_{\sigma}$ is the unique fixed point for the torus action.

\subsection*{Toric maps} Let $X$ and $Y$ be toric varieties corresponding, respectively, to the lattices $N_X$ and $N_Y$ and to the fans 
$\Delta_X$ and $\Delta_Y$. A toric map is a morphism $f\colon X\to Y$ that induces a morphism of algebraic groups $g\colon T_X\to T_Y$ 
such that $f$ is $T_X$-equivariant with respect to the $T_X$-action on $Y$ induced by $g$. Such $f$ corresponds to a unique linear map 
$f_{N_{\real}}\colon (N_X)_{\real}\to (N_Y)_{\real}$ inducing
$f_N\colon N_X\to N_Y$ such that for every cone $\sigma\in\Delta_X$, there is a cone $\tau\in\Delta_Y$ with $f_{N_{\real}}(\sigma)\subseteq \tau$. 
We write  $f_M$ : $M_Y\to M_X$ for the dual of $f_N$.
Note that for $\sigma$ and $\tau$ as above, we have a $k$-algebra homomorphism $k[M_Y\cap\tau^{\vee}]\to k[M_X\cap\sigma^{\vee}]$ mapping $\chi^u$ to $\chi^{f_M(u)}$.
This induces a morphism $U_{\sigma}\to U_{\tau}$, which is the restriction of $f$ to $U_{\sigma}$.  In general, we have $f_{N_{\real}}^{-1}(|\Delta_Y|)\supseteq |\Delta_X|$ and the map $f$ is proper if and only if $f_{N_{\real}}^{-1}(|\Delta_Y|)=|\Delta_X|$. 
Note that by definition, we have a map $f_*\colon\Delta_X\to\Delta_Y$ such that $f_*(\sigma)$ is the smallest cone in $\Delta_Y$ that contains $f_N(\sigma)$. 
It is clear that $f_*$ is a map of posets. 

Suppose now that $f\colon X\to Y$ is a toric map such that the lattice map $f_N\colon N_X\to N_Y$ is surjective. In this case the morphism of tori $T_X\to T_Y$ is surjective.
More generally, if $\sigma$ is a cone in $\Delta_X$, then $f$ induces a surjective morphism of tori $O({\sigma})\to O(\tau)$, where $\tau=f_*(\sigma)$.
We denote the kernel of this morphism by $O(\sigma/\tau)$. Note that this is again a torus, of dimension ${\rm codim}(\sigma)-{\rm codim}(\tau)$. 
It is clear that $f(V({\sigma}))\subseteq V(\tau)$ (with equality if $f$ is proper). 
In fact, the induced map $V(\sigma)\to V(\tau)$ is again a toric map of toric varieties
with the property that the corresponding lattice map is surjective.

\subsection*{Toric Stein factorizations}
Recall that a proper morphism $f\colon X\to Y$ is a \emph{fibration} if $f_*({\mathcal O}_X)={\mathcal O}_Y$. This implies that $f$ is surjective and has connected fibers;
the converse holds if ${\rm char}(k)=0$. The Stein factorization of a proper map $f\colon X\to Y$ is the unique factorization as $X\overset{g}\to Z\overset{h}\to Y$
such that $g$ is a fibration and $h$ is a finite map. 
The following proposition gives the description of fibrations and Stein factorizations in the toric setting.

\begin{pr}\label{Stein_fact}
Let $f\colon X\to Y$ be a proper toric map corresponding to the lattice map $f_N\colon N_X\to N_Y$ and let
$\Delta_X$ and $\Delta_Y$ be the corresponding fans.
\begin{enumerate}
\item[i)] The map $f$ is surjective if and only if ${\rm Coker}(f_N)$ is finite.
\item[ii)] The map $f$ is a fibration if and only if $f_N$ is surjective.
\item[iii)] Suppose that $f$ is surjective. If we let $N_Z=f_N(N_X)$ and $\Delta_Z=\Delta_Y$,
then the factorization $N_X\overset{g_N}\to N_Z\overset{h_N}\to N_Y$ of $f_N$ induces the
Stein factorization of $f$.
\end{enumerate}
\end{pr}

\begin{proof}
Since $f$ is proper, it is surjective if and only if the induced morphism of tori $T_X\to T_Y$ is dominant. This is the case if and
only if $f_M\colon M_Y\to M_X$ is injective, which is equivalent to ${\rm Coker}(f_N)$ being finite. This proves i).

Suppose now that $f_N$ is surjective. In order to show that $f$ is a fibration, it is enough to prove that
for every $\sigma\in\Delta_Y$, the natural map
$$\Gamma(U_{\sigma}, {\mathcal O}_Y)=k[\sigma^{\vee}\cap M_Y]\to \Gamma(f^{-1}(U_{\sigma}), {\mathcal O}_X)$$
is an isomorphism. 
Note that $f^{-1}(U_{\sigma})$ is the union of $U_{\tau}$, where $\tau$ varies over the set $\Lambda_{\sigma}$ of those cones in $\Delta_X$ such that
$f_N(\tau)\subseteq\sigma$. Therefore
$$\Gamma(f^{-1}(U_{\sigma}),{\mathcal O}_X)=\bigcap_{\tau\in\Lambda_{\sigma}}k[\tau^{\vee}\cap M_X]=\bigoplus_{u} k\cdot\chi^u,$$
where the direct sum is over those $u\in M_X$ such that $u\in\tau^{\vee}$ for every $\tau\in \Lambda_{\sigma}$. 
It is enough to show that for every such $u\in M_X$, there is $w\in M_Y\cap \sigma^{\vee}$ such that $f_M(w)=u$
(note that $w$ is clearly unique since $f_M$ is injective). 

It is clear that there is $w\in M_Y$ such that $f_M(w)=u$. Indeed, since $f_{N_{\real}}^{-1}(|\Delta_Y|)=|\Delta_X|$, we deduce that ${\rm Ker}(f_{N_{\real}})$ is a union of cones in $\Delta_X$, which automatically lie in 
$\Lambda_{\sigma}$. Therefore $u\in ({\rm Ker}(f_N))^{\perp}={\rm Im}(f_M)$. We need to show that $w$ lies in $\sigma^{\vee}$. Let $v\in\sigma\cap N_Y$.
Since $f_N$ is surjective, we can write $v=f_N(\widetilde{v})$ for some $\widetilde{v}\in N_X$. 
Since $\widetilde{v}\in f_{N_{\real}}^{-1}(|\Delta_Y|)=|\Delta_X|$, we may choose
$\tau\in \Delta_X$ smallest such that $\widetilde{v}\in\tau$. In this case
$f_{N_{\real}}(\tau)\subseteq\sigma$. Therefore $u\in\tau^{\vee}$, which implies 
$$0\leq \langle u,\widetilde{v}\rangle=\langle f_M(w),\widetilde{v}\rangle=\langle w,f_N(\widetilde{v})\rangle=\langle w,v\rangle$$ and we see that indeed
$w\in\sigma^{\vee}$. 

Suppose now that $f$ is surjective and consider the decomposition $f=h\circ g$ in iii). It is straightforward to see that 
$h$ is finite, while we have already shown that $g$ is a fibration. Therefore this decomposition is the Stein factorization of $f$.
We also deduce from the uniqueness of the Stein factorization that if $f$ is a fibration, then $f_N$ is surjective.
This completes the proof of the proposition.
\end{proof}

\begin{rmk}\label{rmk_fin_map}
A finite, surjective toric morphism $h\colon Z\to Y$ can be described as follows.  
Note that $h_N$ is injective, with finite cokernel.
Suppose first that
${\rm char}(k)$ does not divide the order of $|{\rm Coker}(h_N)|$, which is equal to the order of $A:=M_Z/M_Y$.
In this case, $h$
 is the quotient by a faithful action of $G={\rm Spec}\,k[A]$. In particular, $h$ is generically a Galois cover.
  
Indeed, note first that the assumption on the characteristic of $k$ implies that $G$ is a reduced scheme, isomorphic to the finite group 
 ${\rm Hom}(A, k^*)$. We claim that $G$ acts faithfully on $Z$ such that $Y$ is the quotient by this group action.  
 It is enough to check this on affine open subsets, hence we may assume that $Y=U_{\sigma}$. The morphism $h$ corresponds to
 $$k[\sigma^{\vee}\cap M_Y]\hookrightarrow k[\sigma^{\vee}\cap M_Z].$$
The $G$-action on $Z$ corresponds to 
$$k[\sigma^{\vee}\cap M_Z]\to k[M_Z/M_Y]\otimes_k k[\sigma^{\vee}\cap M_Z],\,\,\chi^u\to \chi^{\overline{u}}\otimes\chi^u,$$
where $\overline{u}$ is the class of $u$ in $M_Z/M_Y$. It is clear from this definition that the action is faithful and furthermore, that 
$$k[\sigma^{\vee}\cap M_Z]^G=k[\sigma^{\vee}\cap M_Y],$$
as claimed.

When ${\rm char}(k)=p>0$, given an arbitrary surjective, finite toric morphism $h\colon Z\to Y$,
we can uniquely factor $h_M\colon M_Y\hookrightarrow M_Z$ as 
$M_Y\hookrightarrow M_{\widetilde{Z}}\hookrightarrow M_Z$, such that the order of $M_Z/M_{\widetilde{Z}}$
is relatively prime to $p$
and $M_{\widetilde{Z}}/M_Y$ is a product of abelian $p$-groups. This corresponds to a factorization of $f$ as $Z\to\widetilde{Z}\overset{\alpha}\to Y$,
with $Z\to\widetilde{Z}$ a quotient as described above and $\alpha$ a universal homeomorphism. In fact, there is $\beta\colon Y\to\widetilde{Z}$ and $m\geq 1$
such that $\beta\circ\alpha={\rm Frob}_{\widetilde{Z}}^m$ and $\alpha\circ\beta={\rm Frob}_Y^m$ (note that a toric variety $W$ is defined over the prime field, hence
it is endowed with a Frobenius morphism ${\rm Frob}_W$ that is linear over the ground field). In order to see this, let us choose a basis $v_1,\ldots,v_n$ for $N_Y$ such that
$p^{m_1}v_1,\ldots,p^{m_n}v_n$ is a basis for $N_{\widetilde{Z}}$, for some positive integers $m_1,\ldots,m_n$. If $m=\max_im_i$ and $Z'$ is the toric variety corresponding to the lattice spanned by  $p^{m_1-m}e_1,\ldots,p^{m_n-m}e_n$
(the fan being the same as that of the toric varieties $Z$, $\widetilde{Z}$, and $Y$), then multiplication by $p^m$ induces an isomorphism of toric varieties $Z'\simeq\widetilde{Z}$. If $\beta\colon Y\to Z'\simeq \widetilde{Z}$
is the induced morphism, then it is easy to check that it has the desired properties.
\end{rmk}

\begin{rmk}\label{rmk_nonsurj}
Suppose that $f\colon X\to Y$ is any proper toric map, possibly not surjective. In this case $f$ has a canonical factorization $X\overset{u}\to W\overset{w}\to Y$
such that both $u$ and $w$ are proper toric maps, with $u$ surjective and $w$ finite, with $T_W\to T_Y$ a closed immersion.
Indeed, if $w\colon W\to Y$ is the normalization of $f(X)$, then since $X$ is normal, there is a unique morphism $u\colon  X\to W$ such that $f=w\circ u$. If
$$N_W:=\{v\in N_Y\mid mv\in f_N(N_X)\,\,\, \text{for some}\,\,\,m\in {\mathbf Z}_{>0}\}$$
and $T$ is the torus corresponding to $N_W$,
then the restriction of $f$ to $T_X$ factors as $T_X\overset{\phi}\to T\overset{\psi}\to T_Y$, with $\phi$ surjective and $\psi$ a closed immersion. 
Therefore $T$ is equal to $f(T_X)$. It is easy to deduce from Chevalley's constructibility theorem that $T$ is an open dense subset of $f(X)$,
hence $T$ admits an open immersion in $W$. Moreover, the map $T\times f(X)\to f(X)$ given by the $T_X$-action on $Y$ induces a 
map $T\times W\to W$ giving an action of $T$ on $W$ that extends the standard action of $T$ on itself. Since by construction $W$ is separated and normal,
it follows that $W$ is a toric variety with torus $T$. Moreover, both $u$ and $w$ are toric maps. 
\end{rmk}

\begin{rmk}\label{factorization_maps_tori}
Let us consider the above decompositions of toric maps in the case of a morphism of algebraic groups $f\colon T_1\to T_2$ between tori. 
We have a decomposition
$$T_1\overset{\phi_1}\to A\overset{\phi_2}\to B\overset{\phi_3}\to C  \overset{\phi_4}\to T_2$$
such that the following hold:
\begin{enumerate}
\item[i)] There is an isomorphism $T_1\simeq A\times A'$, with $A'$ a torus, such that $\phi_1$ corresponds 
to the projection onto the first component.
\item[ii)] $\phi_2$ is finite, surjective, and \'{e}tale, the quotient by the action of a finite group.
\item[iii)] If ${\rm char}(k)=0$, then $\phi_3$ is an isomorphism. If ${\rm char}(k)=p>0$, then there is $\beta\colon C\to B$ and $m\geq 1$
such that $\beta\circ\phi_3={\rm Frob}_{B}^m$ and $\phi_3\circ\beta={\rm Frob}_{C}^m$.
\item[iv)] $\phi_4$ is a closed immersion.
\end{enumerate}
\end{rmk}

\begin{rmk}\label{canonical_fibration}
We say that a toric variety $X$ has \emph{convex, full-dimensional fan support} if $|\Delta_X|$ is a
convex cone in $(N_X)_{\real}$ ${\rm (}$automatically rational polyhedral${\rm )}$, of maximal dimension.
Note that if $X$ 
admits a proper morphism $f\colon X\to Y$, where $Y$ is an affine toric variety of contractible type, then 
$X$ has convex, full-dimensional fan support. Conversely, if this is the case, then 
 we can find $f$ as above.
In fact, we may take $f$ to be a fibration: if  $\Lambda$ is the largest linear subspace contained in $|\Delta_X|$ and $\sigma=|\Delta_X|/\Lambda$, considered as a cone
with respect to the lattice $N_X/N_X\cap \Lambda$, then $Y=U_{\sigma}$ is of contractible type and we have a canonical toric fibration $f\colon X\to Y$. 
Finally, we note that given a proper morphism $f\colon X\to Y$, with $Y$ affine, $f$ is projective if and only if $X$ is quasi-projective.
\end{rmk}

By
Proposition~\ref{Stein_fact} and Remark~\ref{rmk_nonsurj}, we can reduce studying the fibers of an arbitrary proper toric map $f\colon X\to Y$
to the case of a fibration. Because of this, we will mostly consider this case.

\subsection*{Fibers of toric maps} We want to show that the irreducible components of the fibers of a toric map are toric varieties. 
We begin by reviewing some basic facts
that will be  used also later, when reducing to the case of affine toric varieties of contractible type.

Recall that if $X$ is a toric variety defined by the fan $\Delta$ in $N_{\real}$ and $N'$ is a finitely generated subgroup of $N$ such that
$N/N'$ is free and the linear span of $|\Delta|$ is contained in $N'_{\real}$, then we have a toric variety $X'$ corresponding to $\Delta$, considered as a fan in $N'_{\real}$. 
The inclusion $\iota\colon N'\hookrightarrow N$ induces a toric map $X'\to X$. In fact, if we choose a splitting of $\iota$, we get an isomorphism $N\simeq N'\times N/N'$
and a corresponding isomorphism $T_N\simeq T_{N'}\times T_{N/N'}$.
Moreover, we get an isomorphism $X\simeq X'\times T_{N/N'}$ compatible with the decomposition of $T_N$.

A special case that we will often use is that when $X=X(\Delta)$ is an arbitrary toric variety, $\sigma$ is a cone in $\Delta$ and $N'=N_{\sigma}$.
Applying the previous considerations to $U_{\sigma}$, we obtain 
 an isomorphism
of tori and an isomorphism of affine toric varieties
\begin{equation}\label{eq_desc_aff}
T_N\simeq T_{N_{\sigma}}\times O(\sigma),\,\,\,U_{\sigma}\simeq U_{\sigma'}\times O(\sigma),
\end{equation}
where $\sigma'$ is the same as $\sigma$, but considered in $N'_{\real}$.

We now turn to a version of such product decompositions in the relative setting.

\begin{lm}\label{lm_prod_str}
Let $f\colon X\to Y$ be a toric fibration.
Given $\tau\in\Delta_Y$, let us choose a splitting of $(N_Y)_{\tau}\hookrightarrow N_Y$ 
and then a splitting of $f_N$. These determine equivariant isomorphisms 
$$U_{\tau}\simeq U_{\tau'}\times O(\tau)\,\,\,\text{and}\,\,\,f^{-1}(U_{\tau})\simeq f^{-1}(U_{\tau'})\times O(\tau)$$
such that $f^{-1}(U_{\tau})\to U_{\tau}$ gets identified to $f_{\tau'}\times {\rm Id}$, where $f_{\tau'}\colon f^{-1}(U_{\tau'})\to U_{\tau'}$
is the restriction of $f$ over $U_{\tau'}$. Moreover, the following hold:
\begin{enumerate}
\item[i)] We have an induced isomorphism of tori $T_{U_{\tau}}\simeq T_{U_{\tau'}}\times O(\tau)$.
\item[ii)] We have an isomorphism  $f^{-1}(O(\tau))\simeq f^{-1}(x_{\tau})\times O(\tau)$ such that the restriction of $f$ to $f^{-1}(O(\tau))$
corresponds to the projection onto the second component. In particular, $f^{-1}(y)\simeq f^{-1}(x_{\tau})$ for every $y\in O(\tau)$. 
\item[iii)] The map $f_{\tau'}$ is a toric fibration over an affine toric variety of contractible type and we have an isomorphism
$$f^{-1}(x_{\tau})\simeq f_{\tau'}^{-1}(x_{\tau'}).$$ 
\end{enumerate}
\end{lm}

\begin{proof}
Let $N'_Y=(N_Y)_{\tau}$ and $N'_X=f_N^{-1}(N'_Y)$. 
Note that $f^{-1}(U_{\tau})=\bigcup_{\sigma}U_{\sigma}$, where the union is over those $\sigma\in\Delta_X$ such that $f_*(\sigma)\subseteq\tau$.
Therefore the linear span of the support of the fan defining $f^{-1}(U_{\sigma})$ is contained in $(N'_X)_{\real}$. 

The two splittings in the lemma induce isomorphisms 
$$N_Y\simeq N'_Y\times N_Y/N'_Y\,\,\,\text{and}\,\,\,N_X\simeq N'_X\times N_Y/N'_Y$$
such that $f_N$ corresponds to $f'_N\times {\rm Id}$, where $f'_N\colon N'_X\to N'_Y$ is the induced lattice map. 
By applying (\ref{eq_desc_aff}) to both $U_{\sigma}$ and $f^{-1}(U_{\sigma})$ with respect to the subgroups $N'_Y$ and $N'_X$, respectively,
we obtain the assertions in the lemma.
\end{proof}

The following proposition describes the irreducible components for the fibers of proper toric maps.

\begin{pr}\label{str_fibers}
Let $f\colon X\to Y$ be a proper toric map. 
\begin{enumerate}
\item[i)] Every irreducible component of a fiber $f^{-1}(y)$ is a toric variety. Moreover, this is smooth or simplicial
if $X$ has this property. 
\item[ii)] If $f$ is a fibration and $y\in O(\tau)$ for some $\tau\in\Delta_Y$, then $f^{-1}(y)$ is a disjoint union 
of locally closed subsets parametrized by the cones $\sigma\in\Delta_X$ such that $f_*(\sigma)=\tau$,
with the subset corresponding to $\sigma$ being isomorphic to the torus $O(\sigma/\tau)$.
\end{enumerate}
\end{pr}

\begin{proof}
It follows from Proposition~\ref{Stein_fact} and Remark~\ref{rmk_nonsurj} that we may write $f$ as a composition
$X\overset{g}\to Z\overset{h}\to Y$, with $g$ a fibration and $h$ finite. Since a fiber of $f$ is either empty or a disjoint
union of fibers of $g$, it follows that in order to prove i), we may and will assume that $f$ is a fibration.

Let $\tau\in\Delta_Y$ and suppose that $y\in O(\tau)$. 
It follows from Lemma~\ref{lm_prod_str} that $f^{-1}(y)\simeq f^{-1}(x_{\tau})$, hence we may assume that $y=x_{\tau}$. 
Moreover, we have an isomorphism
\begin{equation}\label{eq_str_fib}
f^{-1}(O(\tau))\simeq f^{-1}(x_{\tau})\times O(\tau).
\end{equation} 
It follows that if 
 $W$ is an irreducible component of $f^{-1}(x_{\tau})$, then we get an induced isomorphism $V\simeq W\times O(\tau)$, where $V$ is an irreducible component 
 of $f^{-1}(O(\tau))$. Note that $f^{-1}(O(\tau))$ is preserved by the $T_X$-action, hence $V$ has the same property since $T_X$ is connected. 
 We conclude that the closure $\overline{V}$ 
 is equal to $V(\gamma)$ for some cone $\gamma\in\Delta_X$. 
 Note that $V$ is locally closed in $X$, hence it 
  is open in $\overline{V}$; therefore it is a toric variety with torus $O(\gamma)$.
 Since $O(\gamma)\subseteq f^{-1}(O(\tau))$, we see that $f_*(\gamma)=\tau$. 
Moreover, the isomorphism (\ref{eq_str_fib}) induces an isomorphism 
$O(\gamma)\simeq O(\gamma/\tau)\times O(\tau)$ and it is now clear that 
$W$ is a toric variety with torus $O(\gamma/\tau)$. 
Note that if $X$ is smooth or simplicial, then $V(\gamma)$ has the same property and so does $V$. 
In this case $W$ is smooth, respectively simplicial, as well. This completes the proof of i).

In order to check the assertion in ii), note that $f^{-1}(O({\tau}))$ is the union of  the orbits $O(\sigma)$,
where $\sigma$ runs over the cones of $\Delta_X$ such that $f_*(\sigma)=\tau$. For
every such $\sigma$, the isomorphism (\ref{eq_str_fib}) induces an isomorphism
$O(\sigma)\simeq Z_{\sigma}\times O(\tau)$, such that $Z_{\sigma}$ is isomorphic to $O(\sigma/\tau)$.
Since $f^{-1}(x_{\tau})$ is the disjoint union of the locally closed subsets $Z_{\sigma}$,
this completes the proof of the proposition.   
\end{proof}

\subsection*{The toric Chow lemma}
We will make use of the following toric version of Chow's lemma. For a proof 
in the case when $X$ is complete, see \cite[Theorem~6.1.18]{CLS}. For the general case,
see \cite[Theorem~2]{Sumihiro}.

\begin{pr}\label{Chow}
If $X$ is a toric variety, then there is a projective toric birational morphism $f\colon \widetilde{X}\to X$ such that $\widetilde{X}$
is quasi-projective. 
\end{pr}

\begin{rmk}\label{rmk_toric_resolution}
By combining the toric Chow lemma with toric resolution of singularities, we deduce that for every toric variety $X$, there is a projective birational morphism
$\pi\colon \widetilde{X}\to X$ such that $\widetilde{X}$ is smooth and quasi-projective. 
\end{rmk}

\section{Cohomology of simplicial toric varieties with convex, full-dimensional fan support}

Our goal in this section is to show that if $X$ is a complex simplicial toric variety such that the support  $|\Delta_X|$ is a full-dimensional, convex cone, then the cohomology 
of $X$ behaves similarly to the case when $X$ is complete (we refer to \cite[Chapter~5.2]{Ful93} for that case). Recall that by Remark~\ref{canonical_fibration},
saying that $|\Delta_X|$ is convex and full-dimensional is equivalent to saying that $X$ admits a proper morphism to an affine toric variety, of contractible type.

\subsection*{A good filtration by open subsets}
The key ingredient in this study is a certain filtration of $X$ by open subsets, in the case when $X$ is quasi-projective. 
This filtration has been well-understood and used in the case when $X$ is projective (see for example \cite{Kirwan} and \cite{Fieseler}). We begin by showing that such a filtration also exists when $X$ 
is quasi-projective and has convex, full-dimensional fan support. While we work over ${\mathbf C}$, we remark that the results of this subsection hold over any  algebraically
  closed field.

Let $X$ be a complex quasi-projective toric variety, with convex, full-dimensional fan support. 
We fix an ample torus-invariant Cartier divisor $D$ on $X$ and consider the corresponding polyhedron $P=P_D$
(note that $P$ might not be bounded since $X$ might not be complete). 

We first recall some basic facts related to this setting. Given $D$, to each maximal cone $\sigma\in\Delta_X$ one associates an element $u({\sigma})$
in the lattice $M_X$ such that $D\vert_{U_{\sigma}}={\rm div}(\chi^{-u(\sigma)})$. 
If $P_0$ is the convex hull of the $u(\sigma)$, where 
$\sigma$ varies over the maximal cones of $\Delta_X$, then $P=P_0+|\Delta_X|^{\vee}$. If $Q$ is a face of $P$, then we get a cone $\sigma_Q\in\Delta_X$ 
defined by
$$\sigma_Q=\{w\in (N_X)_{\real}\mid \langle u,w\rangle \leq\langle u',w\rangle\,\,\text{for all}\,\,u\in Q,u'\in P\}.$$
Moreover, each cone in $\Delta_X$ corresponds in this way to a unique face of $P$.
Note that for every maximal cone $\sigma\in\Delta_X$, the corresponding $u({\sigma})$ is a vertex of $P$ and the cone corresponding to $u({\sigma})$ is
$\sigma$, that is, $\sigma^{\vee}$ is generated by $\{u-u(\sigma)\mid u\in P\}$. All these facts are well-known when $X$ is complete. The proofs easily extend to 
our setting, see \cite[Chapter~6]{Mustata}.

Suppose now that
$v\in |\Delta_X|\cap N_X$ is such that the following conditions are satisfied:
\begin{enumerate}
\item[C1)] $v$ is not orthogonal to any minimal generator of the pointed cone $|\Delta_X|^{\vee}$ (equivalently, $v$ does not lie on any facet
of $|\Delta_X|$), and
\item[C2)] The integers
$\langle u(\sigma),v\rangle$, when $\sigma$ varies over the maximal cones in $\Delta_X$, are mutually distinct.
\end{enumerate}
We denote by $\gamma_v\colon \comp^*\to T_X$    the one-parameter subgroup of $T_X$ corresponding to $v$. 
Suppose that $\sigma_1,\ldots,\sigma_r$ are the maximal cones in $\Delta$, ordered such that
\begin{equation}\label{order}
\langle u(\sigma_1),v\rangle<\ldots<\langle u(\sigma_r),v\rangle
\end{equation}
(condition C2) above implies that there is such an ordering and this is of course unique).
For every $i$, with $1\leq i\leq r$, we denote by $x_i$ the torus-fixed point in $U_{\sigma_i}$.

\begin{pr}\label{prop_filtration}
With the above notation, the following hold:
\begin{enumerate}
\item[i)] The only fixed points for the $\comp^*$-action on $X$ induced by $\gamma_v$ are $x_1,\ldots,x_r$.
\item[ii)] For every $i$, let
$X_i$ consist of those $x\in X$ such that the map $\gamma_{v,x}\colon \comp^*\to X$, given by $\gamma_{v,x}(t)=\gamma_v(t)\cdot x$,
extends to a map $\widetilde{\gamma}_{v,x}\colon {\mathbf A}^{\!1}\to X$ with $\widetilde{\gamma}_{v,x}(0)=x_i$. If 
$U_i:=\bigcup_{j\leq i}X_j$, then $U_i$ is open in $X$ for every $i$ and $U_r=X$.
\end{enumerate}
\end{pr}

\begin{proof}
It is clear that $x_1,\ldots,x_r$ are fixed points for the $\comp^*$-action since they are fixed by the $T_X$-action.
Therefore in order to prove i) it is enough to show that if $x\in U_{\sigma_i}$ is a $\comp^*$-fixed point, then $x=x_i$.
By definition, the action of $\gamma_v$ on $U_{\sigma_i}$ is such that
\begin{equation}\label{formula_action}
\chi^u(\gamma_v(t)\cdot x)=t^{\langle u,v\rangle}\chi^u(x)
\end{equation}
for every $t\in \comp^*$ and $u\in\sigma_i^{\vee}\cap M_X$. Since $x$ is a fixed point for the $\comp^*$-action, it follows that
$\langle u,v\rangle=0$ for all $u\in\sigma_i^{\vee}\cap M_X$
such that $\chi^u(x)\neq 0$.

We need to show that
$S:=\{u\in\sigma_i^{\vee}\cap M_X\mid \chi^u(x)\neq 0\}$ is equal to $\{0\}$. We have seen that $S\subseteq v^{\perp}$. 
Note that for $u_1,u_2\in\sigma_i^{\vee}\cap M_X$, we have $u_1+u_2\in S$ if and only if $u_1,u_2\in S$.
Since $\sigma_i^{\vee}$ is generated as a convex cone by $|\Delta_X|^{\vee}$ and by $\{u(\sigma_j)-u(\sigma_i)\mid j\neq i\}$,
it follows that if $S\neq\{0\}$, then either $v$ is orthogonal to a ray of $|\Delta_X|^{\vee}$ or it is orthogonal to
some $u(\sigma_j)-u(\sigma_i)$, with $j\neq i$. Since both these conclusions contradict conditions C1) and C2) above,
we conclude that $S=\{0\}$, completing the proof of i).

We note that for every $x\in X$, the map $\gamma_{v,x}\colon \comp^*\to X$ extends to a map ${\mathbf A}^{\!1}\to X$.
Indeed, consider the canonical toric fibration $f\colon X\to Y$, where $Y$ is an affine toric variety of contractible type (see Remark~\ref{canonical_fibration}).
Since $Y$ is affine and the image of $v$ in $N_Y$ lies in the cone defining $Y$, the composition $f\circ\gamma_{v,x}$ extends to a map
${\mathbf A}^{\!1}\to Y$ (see \cite[Chapter 2.3]{Ful93}). 
Since $f$ is proper, the valuative criterion for properness implies that  $\gamma_{v,x}$ extends to a map $\widetilde{\gamma}_{v,x}\colon {\mathbf A}^{\!1}\to X$. Since $\widetilde{\gamma}_{v,x}(0)$ is clearly fixed
by the $\comp^*$-action induced by $\gamma_v$, it follows from i)  that it is equal to one of the $x_i$. Therefore $\bigcup_{i=1}^rX_i=X$.

Let us describe now $X_i$. If $x\in X_i$, then it is clear that $x\in U_{\sigma_i}$. 
Suppose now that $x\in U_{\sigma_i}$ is arbitrary.
Using again (\ref{formula_action}), we conclude that $\gamma_{v,x}(0)=x_i$
if and only if for every $u\in (\sigma_i^{\vee}\cap M_X)\smallsetminus\{0\}$ with $\chi^u(x)\neq 0$, we have 
$\langle u,v\rangle>0$. 
Let $\tau$ be the face of $\sigma_i$ such that $x\in O({\tau})$.
In this case, for $u\in \sigma_i^{\vee}\cap M_X$ we have $\chi^u(x)\neq 0$ if and only if $u\in\sigma_i^{\vee}\cap\tau^{\perp}\cap M_X$. 
We thus conclude that
$X_i=\bigcup_{\tau}O({\tau})$,
where the union is over the faces $\tau$ of $\sigma_i$ such that 
$\langle u,v\rangle>0$ for every $u\in
(\sigma_i^{\vee}\cap\tau^{\perp}\cap M_X)\smallsetminus\{0\}$.

On the other hand, a face $\tau$ of $\sigma_i$ is of the form $\sigma_Q$ for a unique face $Q$ of $P$ such that $u({\sigma_i})\in Q$.
In this case $\sigma_i^{\vee}\cap\tau^{\perp}$ is generated by $\{u-u(\sigma_i)\mid u\in Q\}$. We deduce that
$X_i$ is the union of those $O({\sigma_Q})$, over the faces $Q$ of $P$ containing $u({\sigma_i})$ having the property that 
$\langle u,v\rangle >\langle u(\sigma_i),v\rangle$ for every $u\in Q$, with $u\neq u(\sigma_i)$. 
On the other hand, since $P=|\Delta_X|^{\vee}+P_0$, we see that for every face $Q$ of $P$, an element $u\in Q$ can be written as
$u=u'+\sum_i\lambda_iu(\sigma_i)$, with $u'\in|\Delta_X|^{\vee}$ and $\lambda_i\in\real_{\geq 0}$, with $\sum_i\lambda_i=1$. 
Since $v\in |\Delta_X|$, it follows from the way we have ordered the maximal cones that 
$$O(\sigma_Q)\subseteq X_i\quad \text{if and only if}\quad i=\min\{j\mid u(\sigma_j)\in Q\}.$$

We conclude that $U_i$ is the union of those orbits $O(\sigma_Q)$ (with $Q$ a face of $P$) such that some $u(\sigma_j)$, with $j\leq i$, lies in $Q$. 
It is clear that if $Q$ has this property, then any face of $P$ that contains $Q$ also has this property. Equivalently, $U_i$ is a union of orbits
such that $O(\sigma)\subseteq U_i$, then $O(\tau)\subseteq U_i$ for every face $\tau$ of $\sigma$. This implies that $U_i$ is open. 
Since $U_r=X_1\cup\ldots\cup X_r=X$,
this completes the proof of ii).
\end{proof}

\begin{cor}\label{cor_filtration}
If $X$ is a quasi-projective toric variety, with convex, full-dimensional fan support, then there are open
torus-invariant subsets 
$$\emptyset=U_0\subseteq U_1\subseteq\ldots\subseteq U_r=X$$
 such that for every $i\geq 1$, the set $U_i\smallsetminus U_{i-1}$ is a closed subset of some $U_{\sigma}$,
where $\sigma$ is a maximal cone in $\Delta_X$. Moreover, if $X$ is smooth (resp. simplicial), then each $U_i\smallsetminus U_{i-1}$
is an affine space (resp. the quotient of an affine space by a finite group).
\end{cor}

\begin{proof}
We consider the $U_i$ given by
Proposition~\ref{prop_filtration}, hence $U_i\smallsetminus U_{i-1}=X_i$. 
We have seen in the proof of the proposition that $X_i\subseteq U_{\sigma_i}$ and 
it follows from definition that 
\begin{equation}\label{eq_cor_filtration}
X_i=\{x\in U_{\sigma_i}\mid\chi^u(x)=0\,\,\text{if}\,\,u\in(\sigma_i^{\vee}\cap M_X)\smallsetminus\{0\}, \langle u,v\rangle\leq 0\}.
\end{equation}
Therefore $X_i$ is closed in $U_{\sigma_i}$. 

Suppose now that $X$ is simplicial and let $v_1,\ldots,v_n$ be the primitive generators of the rays in $\sigma_i$. If $w_1,\ldots,w_n\in M_X$ are
such that $\langle w_j,v_j\rangle>0$ for all $j$ and $\langle w_j,v_k\rangle=0$ for $j\neq k$, then it is easy to deduce from
(\ref{eq_cor_filtration})
that 
$$X_i=\{x\in U_{\sigma_i}\mid \chi^{w_j}(x)=0\,\,\text{if}\,\,\langle w_j,v\rangle\leq 0\}$$
(note that for every nonzero $u\in \sigma_i^{\vee}\cap M_X$, some positive multiple of $u$ can be written as $\sum_ja_jw_j$, with $a_j\in{\mathbf Z}_{\geq 0}$,
and if $\langle u,v\rangle\leq 0$, then there is $j$ with $a_j>0$ and $\langle w_j,v\rangle\leq 0$). This implies that 
$X_i=U_{\sigma_i}\cap V(\tau_i)$, where $\tau_i$ is the face of $\sigma_i$ spanned by those $v_j$, with $j$ such that 
$\langle w_j,v\rangle\leq 0$. In particular, $X_i$ is a simplicial affine toric variety of contractible type (smooth, if $X$ is smooth).
This completes the proof of the corollary.
\end{proof}

\begin{rmk}
The hypothesis that $X$ is quasi-projective is crucial for the construction of the filtration in Corollary~\ref{cor_filtration}. For an example of a complete variety 
for which the construction does not lead to a filtration by open subsets, see \cite{Jur77}.
\end{rmk}

\begin{rmk}
One could construct the filtration in Corollary~\ref{cor_filtration} also following the approach in \cite[Chapter~5.2]{Ful93}. 
\end{rmk}

\subsection*{Applications to cohomology}
 Recall that by work of Deligne \cite{Deligne}, 
for every
complex algebraic variety (assumed to be reduced, but possibly reducible), the cohomology $H^q(X,{\mathbf Q})$ and
$H^q_c(X,{\mathbf Q})$ (singular cohomology and cohomology with compact support) carry natural (rational)  mixed Hodge structures.  A pure Hodge structure of weight $q$
is \emph{of Hodge-Tate type} if all Hodge numbers $h^{i,j}$ vanish, unless $i=j$. In particular, if the underlying vector space is nonzero, then $q$ is even.

We can now give the main results concerning the cohomology of simplicial toric varieties that have convex, full-dimensional fan support.

\begin{tm}\label{BM}
If $X$ is a smooth, quasi-projective, complex toric variety, which has convex, full-dimensional fan support,
then all maps $A_m(X)\to H_{2m}^{BM}(X)$ are isomorphisms, where $A_m(X)$ is the Chow group of $m$-dimensional cycle classes and $H^{BM}_{2m}(X)$ is the $(2m)^{\rm th}$ Borel-Moore homology
group of $X$. If $X$ is not smooth, but simplicial, then the maps are isomorphisms after tensoring with $\rat$.
\end{tm}

\begin{proof}
We omit the proof, as it follows verbatim the one in
\cite[Chapter~5.2]{Ful93}, using Proposition~\ref{cor_filtration}
\end{proof}

\begin{tm}\label{pure_coh_global}
Let $X$ be a simplicial complex toric variety which has convex, full-dimensional fan support.
For every $q$, the mixed Hodge structures on $H^q(X,\rat)$ and $H^q_c(X,\rat)$ are 
pure, of weight $q$, and of Hodge-Tate type. In particular, both $H^*(X,\rat)$ and $H_c^*(X,\rat)$ are even\footnote{We say that a graded vector space
$\bigoplus_{i\in\zed}V^i$ is \emph{even} if $V^i=0$ for all odd $i$.}. 
\end{tm}

We first prove the following lemma.

\begin{lm}\label{lem_surj_coh}
Let $X$ be a smooth, quasi-projective, complex toric variety, with convex, full-dimensional fan support. If $U_0\subseteq\ldots\subseteq U_r=X$
is a filtration as in Corollary~\ref{cor_filtration} and if $Z_i=X\smallsetminus U_i$, then the following hold:
\begin{enumerate}
\item[(a)] If $q$ is odd and $0\leq i\leq r$, then $H^q_c(Z_i,\rat)=0$.
\item[(b)] If $q$ is even and $0\leq i\leq r-1$, then we have an exact sequence
$$0\to H_c^q(Z_i\smallsetminus Z_{i+1},\rat)\to H_c^q(Z_i,\rat)\to H_c^q(Z_{i+1},\rat)\to 0.$$
\end{enumerate}
In particular, for every $q$, the mixed Hodge structure on $H^q_c(X,\rat)$ is pure, of weight $q$, and of Hodge-Tate type.
\end{lm}

\begin{proof}
We prove the assertions in (a) and (b) by descending induction on $i$. For $i=r$, the assertion in (a) is trivial.
We now assume that $(a)$ holds for $i+1$ and show that both (a) and (b) hold for $i$. Consider the long exact sequence
for the cohomology with compact support
$$H^{q-1}_c(Z_{i+1},\rat)\to H^{q}_c(Z_i\smallsetminus Z_{i+1},\rat)\to H^{q}_c(Z_i,\rat)\to H^q_c(Z_{i+1},\rat)\to H^{q+1}_c(Z_i\smallsetminus Z_{i+1},\rat).$$
The key point is that by Corollary~\ref{cor_filtration}, $Z_i\smallsetminus Z_{i+1}$ is isomorphic to an affine space, hence
$H_c^j(Z_i\smallsetminus Z_{i+1},\rat)\simeq \rat$ if $j=2\cdot\dim(Z_i\smallsetminus Z_{i+1})$ and $H_c^j(Z_i\smallsetminus Z_{i+1},\rat)=0$, otherwise. 
The above exact sequence implies that $H^q_c(Z_i,\rat)=0$ if $q$ is odd and, 
using also the fact that (a) holds for $i+1$, we see that
the sequence in (b) is exact when $q$ is even. 

Note that the maps in the exact sequence in (b) are maps of mixed Hodge structures. In particular, they are strict
with respect to both the weight and Hodge filtrations. We thus obtain by descending induction on $i$ that the mixed Hodge structure on $H^q_c(Z_i,\rat)$ is pure, 
of weight $q$, and of Hodge-Tate type (recall that $H^{2d}_c({\mathbf A}^{\!d})$
is pure, of weight $2d$, and of Hodge-Tate type). By taking $i=0$, we obtain the last assertion in the lemma.
\end{proof}

\begin{proof}[Proof of Theorem~\ref{pure_coh_global}]
Note first that since $X$ is simplicial, Poincar\'{e} duality implies that we have an isomorphism of mixed Hodge structures
$$H^q(X,\rat)\simeq H_c^{2d-q}(X,\rat)^*\otimes {\mathbf Q}(-d),$$
where $d=\dim(X)$. This shows that the assertions about $H^*(X,\rat)$ and $H^*_c(X,\rat)$ are equivalent. 
Lemma~\ref{lem_surj_coh} implies that the assertions about $H^*_c(X,\rat)$  hold if $X$ is smooth and quasi-projective, hence we are done in this case.

In the general case, it is enough to show that the assertions about $H^*(X,\rat)$ hold.
We use Chow's lemma and toric resolution of singularities (see Remark~\ref{rmk_toric_resolution}) to get a projective, birational toric map $g\colon \widetilde{X}\to X$
such that $\widetilde{X}$ is smooth and quasi-projective. 
Since the canonical map $g^*\colon H^q(X,\rat)\to H^q(\widetilde{X},\rat)$ is a morphism of mixed Hodge structures (hence strict with respect to both the weight and Hodge filtrations) and since we know that 
$H^q(\widetilde{X},\rat)$ is 
pure, of weight $q$, and of Hodge-Tate type,
it is enough to show that $g^*$ is injective.
 
 In turn, injectivity follows easily from  Poincar\'e duality on the 
$\mathbf Q$-manifold $X$, coupled with the projection formula \ci{iversen}, IX.3.7
($\eta \in H^*(X, \rat)$ and the equality holds in the Borel-Moore homology  of $X$)
\[
g_* \left( \{\widetilde{X}\} \cap g^* \eta \right) =  \left( \{X\} \cap \eta \right).
\]
This completes the proof of the theorem.  We remark that the injectivity of $g^*$
 also follows from the easy-to-prove fact that, in our situation, ${\mathbf Q}_X$
 is a direct summand of $Rg_* \left({\mathbf Q}_{\widetilde{X}} \right).$
 \end{proof}


\begin{rmk}\la{mhs_intcoh}
 The conclusions of Theorem \ref{pure_coh_global} hold for $X$ not necessarily simplicial,
provided we replace cohomology (with compact supports)
 with intersection cohomology (with compact supports). This can be seen  by 
 applying the theorem to a toric resolution of $X$ and by 
 observing that the intersection
cohomology of $X$ is a natural subquotient of the cohomology
of any resolution (see, for example, \ci[Thm. 4.3.1]{decjag}).
\end{rmk}

\subsection*{Betti numbers of simplicial toric varieties with convex, pure-dimensional fan support}
We are now ready to compute the Betti numbers of simplicial toric varieties that admit a proper map to an affine toric variety, of contractible type.
This makes use, as in the complete case (see \cite[Chapter~4.5]{Ful93}), of the Hodge-Deligne polynomial that we now briefly review.

Given a complex algebraic variety $X$ of dimension $n$, one considers the mixed Hodge structure on the groups $H_c^i(X,{\mathbf Q})$.
For every $i$ and $m$, the $m^{\rm th}$ graded piece ${\rm gr}^W_{m}H^i_{c}(X,{\mathbf Q})$ with respect to the weight filtration  is a pure Hodge structure of weight $m$. Therefore the Hodge numbers
$h^{p,q}({\rm gr}^W_{m}H^i_{c}(X,{\mathbf Q}))$ are defined whenever $p+q=m$. The Hodge-Deligne polynomial of $X$ is given by
$$E(X; u,v)=\sum_{p,q\geq 0}e_{p,q}u^pv^q\in{\mathbf Z}[u,v],\,\,\text{where}\,\,e_{p,q}=\sum_{i=0}^{2n} (-1)^ih^{p,q}({\rm gr}^W_{p+q}H^i_{c}(X,{\mathbf Q})).$$

In fact, the Hodge-Deligne polynomial is uniquely characterized by the following two properties:
\begin{enumerate}
\item[a)] If $X$ is smooth and projective, then $E(X; u,v)$ is the usual Hodge polynomial of $X$.
\item[b)] If $Y$ is a closed subset of $X$ and $U=X\smallsetminus Y$, then
$$E(X;u,v)=E(Y;u,v)+E(U;u,v).$$
\end{enumerate}
The Hodge-Deligne polynomial is also multiplicative, that is,
if $X$ and $Y$ are complex algebraic varieties, then 
$$E(X\times Y;u,v)=E(X;u,v)\cdot E(Y;u,v).$$

Due to the additivity property in b) above, it is easy to compute the Hodge-Deligne polynomial of varieties that can be written as disjoint unions
of simple varieties, such as affine spaces or tori. Note that by property a) above, we have $E({\mathbf P}^1; u,v)=uv+1$, hence property b)
implies $E({\mathbf A}^{\!1}; u,v)=uv$ and $E({\mathbf C}^*; u,v)=uv-1$. Using the fact that the Hodge-Deligne polynomial is multiplicative, one obtains
$E({\mathbf A}^{\!n};u,v)=(uv)^n$ and $E(({\mathbf C}^*)^n;u,v)=(uv-1)^n$. 

If $X$ is smooth and projective, of dimension $n$, then  
$$E(X; t,t)=\sum_{i=0}^{2n}(-1)^ib_i(X)t^i,\,\,\,\text{where}\,\,\,b_i(X)=\dim_{{\mathbf Q}}H^i(X,{\mathbf Q}),$$
that is $E(X;t,t)=P_X(-t)$ where $P_X(t)$ is the Poincar\'{e} polynomial. 
This is a consequence of a) and of the Hodge decomposition. 
However, for an arbitrary $X$, we can't recover $\dim_{\mathbf Q}H^i_c(X,{\mathbf Q})$ from the Hodge-Deligne polynomial.
One case when this can be done is when we know that the mixed Hodge structure on each $H^i_c(X,{\mathbf Q})$ is pure of weight $i$. In this case, it follows from the definition
that 
$$E(X;u,v)=\sum_{p,q\geq 0}(-1)^{p+q}h^{p,q}(H^{p+q}_c(X,{\mathbf Q}))u^pv^q$$
and therefore that
$$E(X; t,t)=\sum_{i=0}^{2n}(-1)^i{\rm dim}_{{\mathbf Q}}H^i_c(X,{\mathbf Q})t^i.$$
\begin{rmk}\label{HodgeDeligneforsheaves}
The formalism of the Hodge-Deligne polynomial extends to constructible complexes $K^\bullet$ of sheaves with the property that the stalks of their cohomology sheaves carry a Mixed Hodge structure, or just a weight filtration. In the proof of Theorem~\ref{thm_p_sigma} below we will use this with
$K^\bullet$ the restriction of the intersection complex $IC_X$ of a toric variety $X$ to a union of torus orbits.
\end{rmk}

Let $X$ be an $n$-dimensional toric variety. Let us denote by $d_{\ell}(X)$ the number of cones in $\Delta_X$ of codimension $\ell$.
In other words, $(d_n(X),\ldots,d_0(X))$ is the $f$-vector of $X$ (see \cite{Stanley1, Stanley2}) (the reason we label the entries differently from the usual way is 
for convenience in the relative setting).
Since a toric variety $X$ is the disjoint union of its orbits, we obtain the following well-known formula
\begin{equation}\label{eq_Hodge_Deligne_general}
E(X;u,v)=\sum_{\sigma\in\Delta_X}(uv-1)^{\rm codim(\sigma)}=\sum_{i=0}^nd_i(X)\cdot (uv-1)^i.
\end{equation}

\begin{cor}\label{Betti_convex_support}
Let $X$ be a simplicial toric variety of dimension $n$, with convex, full-dimensional fan support. The odd cohomology (ordinary and with compact supports) vanishes and  we have that
$$\dim_{\rat}H^{2m}_c(X,\rat)=\dim_{\rat}H^{2n-2m}(X,\rat)=\sum_{i=m}^n(-1)^{i-m}{{i}\choose{m}}d_i(X).$$
\end{cor}

\begin{proof}
By Theorem~\ref{pure_coh_global}, 
the mixed Hodge structure on each $H^i_c(X,\rat)$ is pure of weight $i$, hence the formula for the dimension of the cohomology with compact support
follows from (\ref{eq_Hodge_Deligne_general}). The formula for the dimension of the usual cohomology then follows by Poincar\'{e} duality.
\end{proof}

\section{Cohomology of fibers of toric maps with simplicial source}

Our goal in this section is to compute the Betti numbers of the fibers of toric fibrations $f\colon X\to Y$, when $X$ is simplicial. As in the previous section, the main ingredient is the following purity result.

\begin{tm}\label{pure_coh}
Let $f\colon X\to Y$ be a proper toric map between toric varieties, with  $X$ simplicial. For every $y\in Y$ and every $q$, 
the mixed Hodge structure on $H^q(f^{-1}(y),\rat)$ is pure, of weight $q$, and of Hodge-Tate type. In particular, 
$H^*(f^{-1}(y),\rat)$ is even.
\end{tm}

We will deduce the theorem from Theorem~\ref{pure_coh_global}, by showing that if $f\colon X\to Y$ is a toric fibration to an affine toric variety of contractible type,
then the restriction map in cohomology from $X$ to the fiber over the torus-fixed point of $Y$ is an isomorphism. In what follows, we prove this in a slightly
more general setting that we will need in the next section. 

Suppose that $X$ is a complex algebraic variety, considered with the analytic topology. 
Let $G$ be an algebraic group acting
on $X$ and let  ${\rm act}\colon G\times X\to X$ be  the corresponding morphism. By an \emph{equivariant complex of sheaves\footnote{The sheaves are assumed to be sheaves of $\rat$-vector spaces, with respect to the analytic topology.}} on $X$
we mean a complex of sheaves ${\mathcal E}$ on $X$ with the property that there is an isomorphism\footnote{The usual definition also requires the isomorphism to satisfy a cocycle condition. However, we do not need this extra condition.} ${\rm act}^*({\mathcal E})\simeq {\rm pr}_2^*({\mathcal E})$ of complexes on $G\times X$. 
With  slight abuse of language, we say that an object in the derived category of sheaves is equivariant if it can be represented by an equivariant complex.
An important example is that of the constant sheaf $\rat_X$.
It is easy to see that if $f\colon X\to Y$ is a $G$-equivariant morphism and
${\mathcal E}$ is equivariant on $X$, then $Rf_*({\mathcal E})$ is equivariant  on $Y$.

The following lemma was proved in \cite[Lemma~6.5]{DL} in the case of the intersection complex. The general case follows
in the same way, but we reproduce the argument for the benefit of the reader.

\begin{lm}\label{retraction_lemma}
Let $Y=U_{\sigma}$ be an affine toric variety with fixed point $y$. If $v\in N_Y\cap {\rm Int}(\sigma)$ 
and we consider the action of ${\mathbf C}^*$ on $Y$ induced by the 1-parameter subgroup $\gamma_v$, then
for every ${\mathbf C}^*$-equivariant complex ${\mathcal E}$ on $Y$, the natural graded map 
$H^*(Y,{\mathcal E})\to H^*({\mathcal E}_y)$ is an isomorphism.
\end{lm}

\begin{proof}
The assertion in the lemma is equivalent to the fact that $H^*(Y, j_!({\mathcal E}))=0$, where $j\colon Y_0=Y\smallsetminus\{y\}\hookrightarrow Y$ is the inclusion.
The hypothesis on $v$ implies that the map ${\mathbf C}^*\times Y\to Y$ given by $(t,x)\to\gamma_v(t)\cdot x$ extends 
to a map $h\colon {\mathbf A}^{\!1}\times Y\to Y$ such that 
$$h^{-1}(y)=({\mathbf A}^{\!1}\times\{y\})\cup (\{0\}\times Y).$$
Consider the morphism $g\colon Y\to {\mathbf A}^{\!1}\times Y$ given by $g(x)=(1,x)$, so that $h\circ g$ is an isomorphism. 
Therefore the composition
$$H^*(Y,j_!({\mathcal E}))\to H^*({\mathbf A}^{\!1}\times Y,h^*j_!({\mathcal E}))\to H^*(Y, g^*h^*j_!({\mathcal E}))$$
is an isomorphism, hence it is enough to prove that $H^*({\mathbf A}^{\!1}\times Y,h^*j_!({\mathcal E}))=0$. 
On the other hand, we have $h^{-1}(Y_0)={\mathbf C}^*\times Y_0$ and it is easy to see that
$h^*j_!({\mathcal E})\simeq j'_!h_0^*({\mathcal E})$, where $j'\colon {\mathbf C}^*\times Y_0\hookrightarrow {\mathbf A}^1\times Y$ is the inclusion and
$h_0\colon {\mathbf C}^*\times Y_0\to Y_0$ is induced by $h$. Since ${\mathcal E}$ is equivariant with respect to the ${\mathbf C}^*$-action,
it follows that we have an isomorphism $h_0^*({\mathcal E})\simeq {\rm pr}_2^*({\mathcal E})$, hence 
$$j'_!h_0^*({\mathcal E})\simeq j''_!(\rat_{{\mathbf C}^*}\boxtimes j_!({\mathcal E})),$$
where $j''\colon {\mathbf C}^*\hookrightarrow {\mathbf C}$ is the inclusion. Since $H^*({\mathbf A}^{\!1}, j''_!({\mathbf Q}_{{\mathbf C}_*}))=0$, the K\"{u}nneth formula
implies $H^*({\mathbf A}^{\!1}\times Y, h^*j_!({\mathcal E}))=0$.
\end{proof}

\begin{rmk}\label{rmk_retraction_lemma}
Suppose that $f\colon X\to Y$ is a proper surjective toric map, where $Y$ is as in Lemma~\ref{retraction_lemma}. If ${\mathcal E}$ is a $T_X$-equivariant complex on $X$,
then we may choose 
$v$ as in the lemma such that $v$ lies in the image of $f_N$. Therefore we get actions of ${\mathbf C}^*$ on $X$ and $Y$ such that $f$ is ${\mathbf C}^*$-equivariant.
In this case we may apply the lemma to the complex  $Rf_*({\mathcal E})$ and conclude that we have canonical isomorphisms
$$H^i(X,{\mathcal E})\simeq H^i(Y, Rf_*({\mathcal E}))\simeq H^i(Rf_*({\mathcal E})_y)\simeq H^i(f^{-1}(y),{\mathcal E}),$$
where the last isomorphism follows from proper base-change. 
\end{rmk}

We include the next remark for future reference.

\begin{rmk}\label{rmk2_retraction_lemma}
In the setting of Lemma~\ref{retraction_lemma}, we may take ${\mathcal E}={\mathcal I}_Y:=IC_Y[-\dim(Y)]$, where $IC_Y$
is the intersection cohomology complex on $Y$ to conclude
$$IH^i(Y,\rat)\simeq {\mathcal H}^i({\mathcal I}_Y)_y.$$
If $f\colon X\to Y$ is a proper, surjective, toric map, then we may take ${\mathcal E}={\mathcal I}_X$ to conclude
$$IH^i(X,\rat)\simeq H^i(f^{-1}(y),{\mathcal I}_X).$$
\end{rmk}

We can now prove the main result of this section.

\begin{proof}[Proof of Theorem~\ref{pure_coh}]
We first note that we may assume that $f$ is a fibration, $Y$ is affine and of contractible type, and $y\in Y$ is the fixed point.
Indeed, it follows from Proposition~\ref{Stein_fact} and Remark~\ref{rmk_nonsurj} that we may write $f$ as a composition $X\overset{g}\to Z\overset{h}\to Y$,
with $g$ a fibration and $h$ finite. Since every fiber of $f$ is a disjoint union of fibers of $g$, it follows that after replacing $f$ by $g$, we may assume that 
$f$ is a fibration. Suppose now that $\tau\in\Delta_Y$ is such that $y\in O(\tau)$. After replacing $f$ by $f^{-1}(U_{\tau})\to U_{\tau}$, we may assume that $Y=U_{\tau}$. 
Furthermore, by Lemma~\ref{lm_prod_str}, we may assume that the linear span of $\tau$ is the ambient space and that $y=x_{\tau}$. 

In this case we may apply Lemma~\ref{retraction_lemma} to $Rf_*(\rat_X)$ (see Remark~\ref{rmk_retraction_lemma}) to conclude 
that the canonical restriction map
$$H^i(X,\rat)\to H^i(f^{-1}(y),\rat)$$
is an isomorphism. Since this is a morphism of mixed Hodge structures, the assertions in the theorem follow from those in 
Theorem~\ref{pure_coh_global}.
\end{proof}

As in the case of toric varieties, 
it is easy to compute the Hodge-Deligne polynomial of fibers of toric fibrations.
Given a toric fibration $f\colon X\to Y$ 
and a cone $\tau\in\Delta_Y$, we put
$$d_{\ell}(X/\tau)=\#\{\sigma\in\Delta_X\mid f_*(\sigma)=\tau,{\rm codim}(\sigma)-{\rm codim}(\tau)=\ell\}.$$

\begin{rmk}
Note that when $Y$ is a point, then $d_{\ell}(X/\{0\})$ is the same as $d_{\ell}(X)$, as introduced in the previous section. Therefore,
if  the relative dimension of $f$ is $r=\dim(X)-\dim(Y)$, then the vector $(d_r(X/\tau),\ldots,d_0(X/\tau))$
is a relative version of the $f$-vector of a toric variety.
\end{rmk}

\begin{pr}\label{HD_fiber}
Let $X$, $Y$ be toric varieties and let  $f\colon X\to Y$ be  a toric fibration. If $\tau$ is a cone in $\Delta_Y$ and $y\in O(\tau)$, then
$$E(f^{-1}(y); u,v)=\sum_{\ell\geq 0} d_{\ell}(X/\tau)\cdot (uv-1)^{\ell}.$$
In particular, we have
$\chi(f^{-1}(y))=d_0(X/\tau)$.
\end{pr}

\begin{proof}
Recall that by Proposition~\ref{str_fibers}, we can write $f^{-1}(y)$ as a disjoint union of locally closed subsets, with the subsets in one-to-one correspondence
with the cones $\sigma\in\Delta_X$ that satisfy $f_*(\sigma)=\tau$. Moreover, the subset corresponding to $\sigma$ is isomorphic to $({\mathbf C}^*)^{{\rm codim}(\sigma)-{\rm codim}(\tau)}$.
By additivity of the Hodge-Deligne polynomial, we thus get
$$E(f^{-1}(y); u,v)=\sum_{f_*(\sigma)=\tau}(uv-1)^{{\rm codim}(\sigma)-{\rm codim}(\tau)}=\sum_{\ell\geq 0}d_{\ell}(X/\tau)\cdot (uv-1)^{\ell}.$$
The last assertion follows from the usual identity $\chi(f^{-1}(y))=E(f^{-1}(y);1,1)$. 
\end{proof}

\begin{cor}\label{Betti_fib}
Let $X$ and $Y$ be toric varieties, with $X$ simplicial, and let $f\colon X\to Y$ be a toric fibration. 
If $\tau$ is a cone in $\Delta_Y$ and $y\in O(\tau)$, then
$$\dim_{\mathbf Q}H^{2m}(f^{-1}(y),{\mathbf Q})=\sum_{\ell\geq m}(-1)^{\ell-m}{{\ell}\choose m}d_{\ell}(X/\tau).$$

\end{cor}

\begin{proof}
Note that since $f$ is proper, $f^{-1}(y)$ is compact, hence the usual cohomology agrees with the cohomology 
with compact support. By Theorem~\ref{pure_coh}, the mixed Hodge structure on $H^i(f^{-1}(y),{\mathbf Q})$ is pure, hence
$$E(f^{-1}(y);t,t)=\sum_{i\geq 0}(-1)^ib_i(f^{-1}(y))t^i,\quad\text{where}\quad b_i(f^{-1}(y))=\dim_{\mathbf Q}H^{i}(f^{-1}(y),{\mathbf Q}).$$
On the other hand, it follows from Proposition~\ref{HD_fiber} that 
$$E(f^{-1}(y);t,t)=\sum_{\ell\geq 0}d_{\ell}(X/\tau)(t^2-1)^{\ell}
=\sum_{\ell\geq 0}d_{\ell}(X/\tau)\cdot\sum_{m=0}^{\ell}(-1)^{\ell-m}{{\ell}\choose m}t^{2m}$$
$$=\sum_{m\geq 0}\left(\sum_{\ell\geq m}(-1)^{\ell-m}{{\ell}\choose m}d_{\ell}(X/\tau)\right)t^{2m}.$$
The formula in the corollary now follows by equating the coefficients of the powers of $t$ in the two expressions for
$E(f^{-1}(y);t,t)$. 
\end{proof}

\begin{rmk}
If in Corollary~\ref{Betti_fib} we consider the case when $Y$ is a point, then we recover the familiar formula for the Betti numbers of a simplicial,
complete toric variety ${\rm (}$see \cite[Chapter 4.5]{Ful93}${\rm )}$.
\end{rmk}

\begin{rmk}\label{weights_trick}
In fact, one can prove the assertions in Theorems~\ref{pure_coh_global} and \ref{pure_coh} without making use of the filtration constructed
in the previous section, arguing as follows. We have seen that it is enough to prove that when $f\colon X\to Y$ is a proper toric fibration, with $X$
smooth and $Y$ affine, of contractible type, and with $y\in Y$ being the torus-fixed point, then the mixed Hodge structures on $H^i(X,\rat)$ and $H^i(f^{-1}(y),\rat)$
are pure of weight $i$ (the fact that the Hodge structures are of Hodge-Tate type then follows from the computation of the Hodge-Deligne 
polynomials).
 As we have seen, we have an isomorphism of mixed Hodge structures 
$$H^i(X,\rat)\to H^i(f^{-1}(y),\rat).$$
Since $X$ is smooth, all weights on $H^i(X,\rat)$ are $\geq i$, while since $f^{-1}(y)$ is compact, all weights on 
$H^i(f^{-1}(y),\rat)$ are $\leq i$. We conclude that the mixed Hodge structures on each of $H^i(X,\rat)$ and $H^i(f^{-1}(y),\rat)$ are pure of weight $i$.
\end{rmk}

\section{The Decomposition Theorem for toric maps}\la{sec_dec}

The main result of this section is the following version of the Decomposition Theorem in the case of toric fibrations.
Recall that for an algebraic variety $X$, we denote by $IC_X$ the intersection complex on $X$.

\begin{tm}\label{toric_dec_thm}
If $X$ and $Y$ are complex toric varieties and $f\colon X\to Y$ is a toric fibration, then we have a decomposition
\begin{equation}\label{eq_toric_dec_thm}
Rf_*(IC_X)\simeq \bigoplus_{\tau\in\Delta_Y}\bigoplus_{b\in\zed}IC_{V(\tau)}^{\oplus s_{\tau,b}}[-b].
\end{equation} 
Furthermore, the nonnegative integers $s_{\tau,b}$ satisfy the following conditions:

\begin{enumerate}
\item[i)] $s_{\tau,b}=s_{\tau,-b}$ for every $\tau\in\Delta_Y$ and every $b\in\zed$.
\item[ii)] If $f$ is projective, then
$s_{\tau,b}\geq s_{\tau,b+2\ell}$ for every $\tau\in\Delta_Y$ and every $b, \ell \in\zed_{\geq 0}$.
\item[iii)] $s_{\tau,b}=0$ if $b+\dim(X)-\dim(V(\tau))$ is odd.
\end{enumerate}
\end{tm}

With the notation in Theorem~\ref{toric_dec_thm}, for every $\tau\in \Delta_Y$ we put $\delta_{\tau}=\sum_bs_{\tau,b}$. 
It is clear that $\delta_{\tau}$ is a nonnegative integer. The subvariety $V(\tau)$ is a \emph{support} for $f$ if $\delta_{\tau}>0$.
In Sections~\ref{sec_simpl} and \ref{sec_gen} we will give combinatorial descriptions of the invariants $\delta_{\tau}$,
determining in particular the supports of $f$.

\begin{rmk}\label{DT_non_fibration}
Suppose that $X$ and $Y$ are toric varieties and $f\colon X\to Y$ is any proper toric map.
It follows from Proposition~\ref{Stein_fact} and Remark~\ref{rmk_nonsurj} that we can factor $f$ as $X\overset{g}\to Z\overset{h}\to Y$, with
$g$ a toric fibration and $h$ a finite toric map. For every $\tau\in\Delta_Z$, we have an induced finite morphism of algebraic groups $O(\tau)\to O(h_*(\tau))$.
If we denote by $O'(\tau)$ the image of this map, then $O'(\tau)$ is a torus and we have a finite, surjective, \'{e}tale map $h_{\tau}\colon O(\tau)\to O'(\tau)$, which is the quotient
by a finite group ${\rm (}$see Remark~\ref{factorization_maps_tori}${\rm )}$. It follows that if ${\mathcal L}_{\tau}:=(h_{\tau})_*({\mathbf Q}_{O(\tau)})$, then ${\mathcal L}_{\tau}$ is a local system on $O'(\tau)$.
Moreover, the induced map $\overline{h}_\tau:V(\tau)\to\overline{O'(\tau)}$ is finite, hence the direct image by this map is $t$-exact for the middle perversity $t$-structure, and therefore
it preserves intersection complexes with twisted coefficients. This implies that ${\overline{h}_\tau}_*(IC_{V(\tau)})\simeq IC_{\overline{O'(\tau)}}({\mathcal L}_{\tau})$. We thus obtain from Theorem~\ref{toric_dec_thm}
the following decomposition in this general toric setting:
$$Rf_*(IC_X)\simeq \bigoplus_{\tau\in\Delta_Z}\bigoplus_{b\in\zed}IC_{\overline{O'(\tau)}}^{\oplus s_{\tau,b}}({\mathcal L}_{\tau})[-b].$$
Clearly, the supports of $f$ are the images via $h$ of the supports
of the toric fibration $g.$
\end{rmk}

Before giving the proof of Theorem~\ref{toric_dec_thm} we make some preparations. We begin by recalling the following general statement
of the Decomposition Theorem, see \cite[Th\'eor\`eme 6.2.5]{BBD}.

\begin{tm}\label{gen_dec_thm}
Let $X$ and $Y$ be complex algebraic varieties and consider a proper morphism $f\colon X\to Y$.
We have a  finite direct sum decomposition
$$Rf_*(IC_X)\simeq \bigoplus_{\alpha}IC_{\overline{Y_{\alpha}}}(L_{\alpha})[-d_{\alpha}],$$
where each $Y_{\alpha}$ is a smooth, irreducible, locally closed subset of $Y,$ $L_{\alpha}$ is a local system on
$Y_{\alpha}$  and $d_\alpha \in \zed.$
\end{tm}

In order to prove Theorem~\ref{toric_dec_thm}, we will need to show that under the assumptions of the theorem,
all varieties $\overline{Y_{\alpha}}$ are torus-invariant and the local systems $L_{\alpha}$ are trivial. 
Properties i) and ii) will then follow from Poincar\'{e} duality and relative Hard Lefschetz. Finally, property iii) will
be a consequence of the following proposition.

\begin{pr}\label{even_IC}
If $f\colon X\to Y$ is a proper toric map, then for every $y\in Y$, we have $H^i(f^{-1}(y),IC_X)=0$ whenever $i+\dim(X)$ is odd.
\end{pr}

\begin{proof}
By Proposition~\ref{Stein_fact} and Remark~\ref{rmk_nonsurj}, we may factor $f$ as a composition $X\overset{g}\to Z\overset{h}\to Y$, with $g$ a toric fibration and $h$ finite.
Since a fiber of $f$ is a disjoint union of fibers of $g$, it follows that we may assume that $f$ is a fibration.
Let $\pi\colon X'\to X$ be a 
proper, birational, toric morphism such that $X'$ is smooth and let $f'=f\circ\pi$. Since $\pi$ is birational,
it follows from 
Theorem~\ref{gen_dec_thm} that 
$IC_X$ is a direct summand of $R\pi_*(IC_{X'})$. By restricting over $f^{-1}(y)$, applying base-change, and taking
the $i^{\rm th}$ cohomology, we conclude that 
$H^i(f^{-1}(y),IC_X)$ is a summand of 
$$H^i(f^{-1}(y), R\pi_*(IC_{X'}))=H^i({f'}^{-1}(y),IC_{X'})=H^{i+n}({f'}^{-1}(y),\rat),$$
where $n=\dim(X)$ and the second equality follows from the fact that $X'$ being smooth, we have $IC_{X'}=\rat_{X'}[n]$.
On the other hand, since $X'$ is smooth and $f\circ\pi$ is a fibration, we may apply Theorem~\ref{pure_coh} to get $H^{i+n}({f'}^{-1}(y),\rat)=0$
when $i+n$ is odd. This completes the proof of the proposition.
\end{proof}

\begin{rmk}\label{mhs_Saito}
It follows from Saito's theory of mixed Hodge modules \cite{Saito} that for every proper morphism $f\colon X\to Y$ of complex algebraic varieties
and every $y\in Y$, the cohomology groups $H^{i-\dim X}(f^{-1}(y),IC_X)$
$=H^{i}(f^{-1}(y),\mathcal{I}_X)$ carry a natural mixed Hodge structure. The argument in the proof of Theorem~\ref{even_IC}
together with Theorem~\ref{pure_coh} imply that if $f$ is a toric map, then this mixed Hodge structure is pure of weight $i$.
\end{rmk}

The fact that in the toric decomposition theorem we only have torus-invariant subvarieties as supports will follow
from the next lemma. Given a toric variety $Y$, we denote by $\Omega_Y$ the stratification by the $T_Y$-orbits.
Recall that a complex of sheaves ${\mathcal E}$ on $Y$ is \emph{$\Omega_Y$-constructible} if for every $i$ and every $O\in \Omega_Y$,
the restriction ${\mathcal H}^i({\mathcal E})\vert_O$ is a local system on $O$.  The intersection complex
of a toric variety $Y$ is $\Omega_Y$-constructible in a strong sense, i.e.  each of these restrictions
is a constant (torus) equivariant sheaf on the orbit (see \ci[Lemma~5.15]{BL}). The same is true for the direct image complex  in (\ref{eq_toric_dec_thm}), and the lemma that follows is a step in proving Theorem \ref{toric_dec_thm}.

\begin{lm}\label{Omega_constructibility}
If $f\colon X\to Y$ is a proper toric fibration, then $Rf_*(IC_X)$ is $\Omega_Y$-constructible. In fact, the
restriction of each ${\mathcal H}^i(Rf_*(IC_X))$ to a $T_Y$-orbit is a constant sheaf.
In particular, $IC_Y$ is $\Omega_Y$-constructible. 
\end{lm}

\begin{proof} 
Let $\tau\in\Delta_Y$ be fixed.
In order to describe the restriction of some ${\mathcal H}^i(Rf_*(IC_X))$ to the orbit $O(\tau)$, we may restrict to the affine open subset $U_{\tau}$ and thus assume that $Y=U_{\tau}$.
Lemma~\ref{lm_prod_str} implies that we have isomorphisms $Y\simeq Y'\times O(\tau)$ and $X\simeq X'\times O(\tau)$,
with $Y'$ having a fixed point $y$, 
such that $f$ gets identified to $f'\times {\rm Id}_{O(\tau)}$, where $f'\colon X'\to Y'$ is a toric fibration. 
In this case $IC_{X}$ gets identified to ${\rm pr}_1^*(IC_{X'})[r]$, where $r=\dim(O(\tau))$. 
It is then clear that
$${\mathcal H}^i(Rf_*(IC_X))\vert_{O(\tau)}\simeq {\mathcal H}^{i+r}(Rf'_*(IC_{X'}))_y\otimes \rat_{O(\tau)}.$$
This gives the  first two assertions in the lemma. The third is the special case
$f={\rm Id}.$
\end{proof}

We can now prove the main result of this section.

\begin{proof}[Proof of Theorem~\ref{toric_dec_thm}]
Given a toric fibration $f\colon X\to Y$, we obtain via Theorem~\ref{gen_dec_thm} a finite direct sum decomposition
$$Rf_*(IC_X)\simeq \bigoplus_{\alpha}IC_{\overline{Y_{\alpha}}}(L_{\alpha})[-d_{\alpha}],$$
with each $Y_{\alpha}$ a smooth, irreducible, locally closed subset of $Y$ and $L_{\alpha}$ a local system on $Y_{\alpha}$. 
We may assume that each $L_{\alpha}$ is nonzero and indecomposable.

Since $Rf_*(IC_X)$ is $\Omega_Y$-constructible  by 
Lemma~\ref{Omega_constructibility}, so is each term
$IC_{\overline{Y_{\alpha}}}(L_{\alpha}).$ It follows that
$\overline{Y_{\alpha}}= V(\sigma_\alpha)$ for a unique cone $\s_\alpha \in \Delta_Y$ and that
we can take $Y_\alpha= O(\s_\alpha).$

The proof of the same lemma implies that the restriction of ${\mathcal H}^{-\dim {O(\s_\alpha)}}(IC_{\overline{V(\s_\alpha)}} (L_{\alpha}))$ to
$ O(\sigma_\alpha)$ is constant.  This  restriction is $L_\alpha,$ which is thus
 constant, hence isomorphic to $\rat_{O(\s_{\alpha})}$ (being indecomposable). Therefore $IC_{\overline{Y_{\alpha}}}(L_{\alpha})=IC_{V(\sigma_\alpha)}$.

The assertions in i) and ii) now follow from Poincar\'{e} duality and relative Hard Lefschetz (see \cite[Th\'eor\`eme 5.4.10]{BBD}). In order to check the assertion in iii),
we consider for $\tau\in \Delta_Y$ the stalk at $x_{\tau}$ for both sides of (\ref{eq_toric_dec_thm}). By taking the $i^{\rm th}$ cohomology
and applying base-change, we conclude that ${\mathcal H}^{i-b}(IC_{V(\tau)})_{x_{\tau}}^{\oplus s_{\tau,b}}$ is a direct summand of
$H^i(f^{-1}(x_{\tau}),IC_X)$. 
If we have $b+\dim(X)-\dim(V(\tau))$ odd, then Lemma~\ref{even_IC} implies that $H^i(f^{-1}(x_{\tau}),IC_X)=0$ when $i=b-\dim(V(\tau))$. However, in this case
${\mathcal H}^{i-b}(IC_{V(\tau)})_{x(\tau)}\simeq \rat$, which implies $s_{\tau,b}=0$.
\end{proof}

\begin{rmk}\label{rmk_KS1}
As we have mentioned in the Introduction, Katz and Stapledon associate some invariants, \emph{local h-polynomials} to certain maps between posets.
This is done in a purely combinatorial way. Given a toric fibration $f\colon X\to Y$, their framework can be applied to the map $f_*\colon\Delta_X\to\Delta_Y$
and it turns out that our invariants $s_{\tau,b}$ can be interpreted as coefficients of the local $h$-polynomial.
Due to the combinatorial definition, the nonnegativity of these coefficients is not a priori clear. However, one of the main results in \cite{KS} says that in the case
of a regular rational polyhedral subdivision of a rational polytope ${\rm (}$which corresponds to a projective toric birational morphism between  toric varieties${\rm )}$, the coefficients of the local $h$-polynomial are nonnegative,
symmetric, and unimodal ${\rm (}$see \cite[Theorem 6.1]{KS}${\rm )}$.
\end{rmk}

\section{Combinatorics of the toric Decomposition Theorem: the simplicial case}\la{sec_simpl}

In this section we study toric fibrations $f\colon X\to Y$, with $X$ and $Y$ simplicial toric varieties. In this case
we can determine explicitly all multiplicities $s_{\tau,b}$ that appear in Theorem~\ref{toric_dec_thm}.

We begin by recalling some elementary combinatorial definitions and facts about the incidence algebra associated with a partially ordered set. 
For an introduction to incidence algebras and applications, see \cite{Rota64} or \cite[\S 3.6]{Stanley4}. 
If $({\mathcal P},\leq)$ is a finite poset and $K$ is a commutative ring with identity, we have the \emph{incidence algebra} ${\mathbb I}({\mathcal P},K)$ 
consisting of the set of functions $$f\colon \{(a,b)\in {\mathcal P}\times {\mathcal P}\mid a\leq b\}\to K$$ with a \emph{convolution} operation defined by 
$$f\star g (a,b):=\sum_{a\leq c\leq b}f(a,c)\cdot g(c,b).$$
This is an associative operation with an identity element given by the \emph{ delta function} $\delta$\footnote{This should not be confused with our key invariant $\delta$, a function of a single variable.}, with $\delta(a,b)=1$ if $a=b$ and $\delta(a,b)=0$, otherwise. 
It is easy to check, see \cite[Proposition 3.6.2]{Stanley4}, that  $f$ has a (unique) inverse with respect to convolution if and only if $f(a,a)$ is a unit of $K$ for every $a\in {\mathcal P}$ (the proof in  \cite{Stanley4}, given for $K=\zed$, holds without any change for any commutative ring with identity). 

 An important example is that 
of the function $\zeta$ given by $\zeta(a,b)=1$ whenever $a\leq b$. The inverse of $\zeta$ with respect to convolution is the \emph{M\"{o}bius function} $\mu_{\mathcal P}$
of ${\mathcal P}$. For example, if $({\mathcal P},\leq)$ is the set of all subsets of a finite set, ordered by inclusion, then it is easy to see that
$\mu_{\mathcal P}(A,B)=(-1)^{\#(B\smallsetminus A)}$ whenever $A\subseteq B$.

Another fact that follows easily from definition is that if $f$ is a function as above with inverse $g$ with respect to convolution and
$\phi,\psi\colon {\mathcal P}\to K$ are functions such that
$\phi(x)=\sum_{y\leq x}f(y,x)\psi(y)$ for every $x\in {\mathcal P}$, 
then $\psi(y)=\sum_{x\leq y}g(x,y)\phi(x)$ for every $x\in {\mathcal P}$. When $f=\delta$, this is the \emph{M\"{o}bius inversion formula}.

After these preparations, we return to toric maps. In the following theorem we consider a toric fibration $f\colon X\to Y$, in which both $X$ and $Y$ are simplicial. In this case the decomposition
(\ref{eq_toric_dec_thm}) becomes easier to describe, since $IC_X={\mathbf Q}_X[\dim(X)]$ and $IC_{V(\tau)}=\rat_{V(\tau)}[\dim(V(\tau))]$ for every cone $\tau\in\Delta_Y$
(we use the fact that both $X$ and $V(\tau)$ are $\rat$-manifolds).
Recall that if $f\colon X\to Y$ is a toric fibration, then for every $\tau\in\Delta_Y$ we put
$$d_{\ell}(X/\tau)=\#\{\alpha\in\Delta_X\mid f_*(\alpha)=\tau, {\rm codim}(\alpha)-{\rm codim}(\tau)=\ell\}.$$

\begin{tm}\label{form_both_simplicial}
Suppose we are in the setting of Theorem~\ref{toric_dec_thm}, with both $X$ and $Y$  simplicial.
\begin{enumerate}
\item[i)] For every $\tau\in\Delta_Y$,  we have
 $$\delta_{\tau}:=\sum_bs_{\tau,b}=\sum_{\sigma\subseteq\tau}(-1)^{\dim(\tau)-\dim(\sigma)}d_0(X/\sigma).$$
\item[ii)] For every $m\in\zed$ and every $\tau\in\Delta_Y$, we have
$$s_{\tau,2m +\dim(V(\tau))-\dim(X)}=\sum_{\sigma\subseteq\tau}(-1)^{\dim(\tau)-\dim(\sigma)}\cdot \sum_{\ell\geq m}(-1)^{\ell-m}{{\ell}\choose m}d_{\ell}(X/\sigma)$$
${\rm (}$where the right-hand side is understood to be $0$ if $m<0$${\rm )}$, while $s_{\tau,i+\dim(V(\tau))-\dim(X)}=0$ if $i$ is odd.
\end{enumerate}
\end{tm}

\begin{proof}
Let $d_X=\dim(X)$. 
Since $X$ and $Y$ are simplicial, it follows from Theorem~\ref{toric_dec_thm} that we have a decomposition
\begin{equation}\label{eq_form_both_simplicial}
Rf_*(\rat_X[d_X])\simeq\bigoplus_{\tau\in\Delta_Y}\bigoplus_{b\in\zed}\rat_{V(\tau)}^{\oplus s_{\tau,b}}[\dim(V(\tau))-b].
\end{equation}
If $\sigma\in\Delta_Y$, by taking the stalk at $x_{\sigma}$ and computing the $(i-d_X)^{\rm th}$ cohomology, we obtain
via base-change
\begin{equation}\label{eq3_form_both_simplicial}
H^{i}(f^{-1}(x_{\sigma}),\rat)\simeq \bigoplus_{\tau\subseteq\sigma}\rat^{\oplus s_{\tau,i+\dim(V(\tau))-d_X}}.
\end{equation}
The second assertion in ii) follows directly from Theorem~\ref{toric_dec_thm}, while
Theorem~\ref{pure_coh} implies
$H^j(f^{-1}(x_{\sigma}),\rat)=0$ for $j$ odd. We conclude that 
$$\chi(f^{-1}(x_{\sigma}))=\sum_{i\geq 0}\dim_{\rat}H^i(f^{-1}(y),\rat)=\sum_{\tau\subseteq\sigma}\delta_{\tau}.$$
The M\"{o}bius inversion formula for the poset $\Delta_Y$ implies
\begin{equation}\label{eq2_form_both_simplicial}
\delta_{\tau}=\sum_{\sigma\subseteq\tau}\mu_{\Delta_Y}(\sigma,\tau)\chi(f^{-1}(x_{\sigma}))=\sum_{\sigma\subseteq\tau}\mu_{\Delta_Y}(\sigma,\tau)d_0(X/\sigma),
\end{equation}
where the second equality follows from Proposition~\ref{HD_fiber}. 
On the other hand, $\mu_{\Delta_Y}(\sigma,\tau)$ only depends on the interval $[\sigma,\tau]$ in $\Delta_Y$ and since $\tau$ is simplicial, this interval 
is in order-preserving bijection 
with the poset of all subsets of a set with $\dim(\tau)-\dim(\sigma)$ elements. Therefore $\mu_{\Delta_Y}(\sigma,\tau)=(-1)^{\dim(\tau)-\dim(\sigma)}$. 
The formula in (\ref{eq2_form_both_simplicial}) thus gives the assertion in i).

We proceed similarly to prove ii). Let $m$ be a fixed integer.
It follows from 
(\ref{eq3_form_both_simplicial}) that for every $\sigma\in\Delta_Y$ we have
$$\dim_{\rat}H^{2m}(f^{-1}(x_{\sigma}),\rat)=\sum_{\tau\subseteq\sigma}s_{\tau,2m+\dim(V(\tau))-d_X}.$$
The M\"obius inversion formula and Corollary~\ref{Betti_fib} imply
$$s_{\tau,2m+\dim(V(\tau))-d_X}=\sum_{\sigma\subseteq\tau}(-1)^{\dim(\tau)-\dim(\sigma)}\dim_{\rat}H^{2m}(f^{-1}(x_{\sigma}),\rat)$$
$$=
\sum_{\sigma\subseteq\tau}(-1)^{\dim(\tau)-\dim(\sigma)}\cdot \sum_{\ell\geq m}(-1)^{\ell-m}{{\ell}\choose m}d_{\ell}(X/\sigma).$$
This completes the proof of the theorem.
\end{proof}

\begin{rmk}\label{rmk_KS2}
The reader can compare the formula for the invariants $s_{\tau,b}$ in Theorem~\ref{form_both_simplicial} with the formula in \cite[Lemma~4.12]{KS} for the local $h$-polynomial
of a map of posets $\Gamma\to B$, in which $\Gamma$ is simplicial and $B$ is a Boolean algebra. 
\end{rmk}

\begin{rmk}\la{iop}
Let $f\colon X\to Y$ be a toric fibration, with both $X$ and $Y$ simplicial. It follows from Theorem~\ref{form_both_simplicial}
that for every $\tau\in\Delta_Y$, the expression 
$$\sum_{\sigma\subseteq\tau}(-1)^{\dim(\tau)-\dim(\sigma)}d_0(X/\sigma)$$
is nonnegative ${\rm (}$and it is positive if and only if $V(\tau)$ is a support for $f$${\rm )}$. 
We do not know a direct combinatorial argument that would imply that the expression is nonnegative. 
A similar remark can be made in the not-necessarily-simplicial case, following 
the combination of Theorems \ref{form_general} and \ref{thm_p_sigma}. 
In Remark~\ref{rmk_special_case} below we give such an argument when $f$ is birational  between simplicial toric varieties   and $\dim(\tau)\leq 3$.  
\end{rmk}

\begin{rmk}\label{rmk_rel_f_vector}
Note that the invariants $s_{\tau,b}$ in Theorem~\ref{toric_dec_thm} satisfy the conditions ${\rm i)}$ and ${\rm ii)}$ coming from Poincar\'{e} duality
and relative Hard Lefschetz.  
In the setting of Theorem~\ref{form_both_simplicial}, these translate into interesting conditions satisfied by the invariants $d_{\ell}(X/\tau)$, for
the cones $\tau\in\Delta_Y$. More precisely, suppose that $f\colon X\to Y$ is a projective toric fibration between simplicial toric varieties. 
For a cone $\sigma\in\Delta_Y$ and $m\geq 0$, let us put
$$\widetilde{d}_m(X/\sigma)=\sum_{\ell\geq m}(-1)^{\ell-m}{\ell\choose m}d_{\ell}(X/\sigma).$$
Recall that by Corollary~\ref{Betti_fib}, we have $\widetilde{d}_m(X/\sigma)=\dim_{\rat}H^{2m}(f^{-1}(y);\rat)$ for any $y\in O(\sigma)$. In particular,
we have $\widetilde{d}_m(X/\sigma)\geq 0$.
With this notation, Poincar\'{e} duality says that for every $m$ with 
$0\leq m\leq \dim(X)-\dim(V(\tau))$, if $m'=\dim(X)-\dim(V(\tau))-m$, then
$$\sum_{\sigma\subseteq\tau}(-1)^{\dim(\tau)-\dim(\sigma)}\widetilde{d}_m(X/\sigma)=
\sum_{\sigma\subseteq\tau}(-1)^{\dim(\tau)-\dim(\sigma)}\widetilde{d}_{m'}(X/\sigma).$$
Similarly, relative Hard Lefschetz says that if $0\leq m\leq \frac{1}{2}(\dim(X)-\dim(V(\tau)))$, then
$$\sum_{\sigma\subseteq\tau}(-1)^{\dim(\tau)-\dim(\sigma)}\widetilde{d}_m(X/\sigma)\geq \sum_{\sigma\subseteq\tau}(-1)^{\dim(\tau)-\dim(\sigma)}\widetilde{d}_{m+1}(X/\sigma).$$
These conditions generalize to the relative setting the famous restrictions on the $f$-vector of a simplicial toric variety that come from Poincar\'{e} duality
and Hard Lefschetz ${\rm (}$see \cite[Chapter~5.6]{Ful93}${\rm )}$.
\end{rmk}

\begin{rmk}\label{rmk_special_case}
Suppose that $f\colon X\to Y$ is a proper, birational toric map. We may assume that $N_X=N_Y$ and $f_N$ is the identity, hence
$\Delta_X$ gives a fan refinement of $\Delta_Y$. If for a cone $\tau\in\Delta_Y$ we define $\delta_{\tau}$ by 
\begin{equation}\label{eq_rmk_special_case}
\delta_{\tau}=\sum_{\sigma\subseteq\tau}(-1)^{\dim(\tau)-\dim(\sigma)}d_0(X/\sigma),
\end{equation}
then we want to give a ``nonnegative expression" for $\delta_{\tau}$. For every cone $\tau\in\Delta_Y$,
let $\iota(\tau)$ denote the number of rays in $\Delta$ that are contained in the relative interior of $\tau$. 
If $\dim(\tau)\leq 3$, then we have the following formulas:
\begin{enumerate}
\item[i)] $\delta_0=1$ and $\delta_{\tau}=0$ if $\dim(\tau)=1$.
\item[ii)] $\delta_{\tau}=\iota(\tau)$ if $\dim(\tau)=2$.
\item[iii)] $\delta_{\tau}=2\iota(\tau)$ if $\dim(\tau)=3$. 
\end{enumerate}
The assertions in ${\rm i)}$ and ${\rm ii)}$ follow easily from (\ref{eq_rmk_special_case}), hence we only prove ${\rm iii)}$. 
In order to check this, it is convenient to consider a transversal section $T$ of $\tau$. This is a triangle such that
$\Delta_X$ induces a triangulation $\Lambda$ of $T$. Let us consider the following invariants:
\begin{enumerate}
\item[1)]  $a_3$ is the number of triangles in $\Lambda$,
\item[2)] $a_2$ is the number of segments in $\Lambda$ that are contained in the boundary of $T$,
\item[3)]  $a'_2$ is the number 
of segments in $\Lambda$ not contained in the boundary of $T$,
\item[4)]  $a_1$ is the number of points in $\Lambda$ in the boundary of $T$,
\item[5)] $a'_1$ is the number of points in the interior of $T$ ${\rm (}$hence $a'_1=\iota(\sigma)$${\rm )}$.
\end{enumerate}
Note that we have the following relations between these invariants:
\begin{enumerate}
\item[R1)] $(a_1+a'_1)-(a_2+a'_2)+a_3=1$ ${\rm (}$by considering the Euler-Poincar\'{e} characteristic of $T$${\rm )}$.
\item[R2)] $a_1=a_2$.
\item[R3)] $3a_3=a_2+2a'_2$ ${\rm (}$by counting the segments in the boundaries of all triangles, and noting that a segment appears in
2 triangles if it is not contained in the boundary of $T$, and in 1, otherwise${\rm )}$.
\end{enumerate}
By combining ${\rm R1)}$ and ${\rm R3)}$, we see that
$$3a_2+3a'_2-3a_1-3a'_1+3=a_2+2a'_2.$$
Simplifying and using also ${\rm R2)}$, we obtain:
\begin{equation}\label{eq3_rmk_special_case}
a'_2-a_1-3a'_1+3=0.
\end{equation}
On the other hand, it follows from ${\rm (}$\ref{eq_rmk_special_case}${\rm )}$ that
$$\delta_{\tau}=a_3-a_2+2.$$
By using ${\rm R1)}$ and ${\rm (}$\ref{eq3_rmk_special_case}${\rm )}$, we obtain the desired conclusion:
$$\delta_{\tau}=a'_2-(a_1+a'_1)+3=3a'_1-a'_1=2a'_1=2\iota(\sigma).$$
 It is worth noting that if $\dim(\tau)=4$, then $\delta_{\tau}$ is not a multiple of $\iota(\tau)$.
 Indeed, by considering the blow-up of ${\mathbf A}^{\!4}$ at the origin, we see that the only possibility would be
 $\delta_{\tau}=3\iota(\tau)$. On the other hand, consider $f=g\circ h$, where $g\colon Z\to {\mathbf A}^{\!4}$ is the blow-up of an invariant line $L$, 
 with exceptional divisor $E\simeq  {\mathbf P}^2\times L$ and $h$ is the blow-up of $Z$ along the subset ${\mathbf P}^2\times\{0\}\subset E$.
 An easy computation shows that in this case $\iota(\tau)=1$ but $\delta(\tau)=4$. 
\end{rmk}

\section{Combinatorics of the toric Decomposition Theorem: the general case}\la{sec_gen}

Our  goal in this section is to determine the supports of an arbitrary toric fibration $f\colon X\to Y$
and to show that they are combinatorially determined. 
In this case there are two difficulties, compared with the setting in the previous section: on one hand, the poset structure of $\Delta_Y$ is more complicated; second,
and more crucially, we need to take into account the 
singularities of $X$ and $Y$. These will come up through the local behavior of the intersection cohomology complexes.
In order to deal with the latter issue we begin by introducing the following invariant of an arbitrary toric variety.

Let $\zed [T,T^{-1}]$ denote the ring of Laurent polynomials with integer coefficients.
Given a toric variety $Y$ and two cones $\tau\subseteq\sigma$ in $\Delta_Y$, we define
$$R_{\tau, \sigma}(T)= \sum_{k \in \zed}  \dim_{\rat}{\mathcal H}^{k}(IC_{V(\tau)})_{x_{\sigma}}T^k \in \zed [T,T^{-1}]
$$
and 
$$   r_{\tau,\sigma}=R_{\tau, \sigma}(1)= \sum_{k\in\zed} \dim_{\rat}{\mathcal H}^{k}(IC_{V(\tau)})_{x_{\sigma}}.$$

Note that since the restriction of ${\mathcal H}^k(IC_{V(\tau)})$ to each torus-orbit is constant by Lemma~\ref{Omega_constructibility}, we could have replaced in the above definition
$x_{\sigma}$ by any other point in $O(\sigma)$.

 \begin{rmk}\label{rmk_combinatorial_invariant}
The function $R\colon \{(\tau,\sigma)\in\Delta_Y\times\Delta_Y\mid \tau\subseteq\sigma\}\to \zed[T,T^{-1}]$ only depends on the combinatorics of $\Delta_Y$.
Indeed, in order to see that $R_{\tau,\sigma}$ is combinatorially determined, we may replace $Y$ by $V(\tau)$ and thus assume that $\tau=\{0\}$.
In this case, the assertion is a consequence of \cite[Theorems 1.1, 1.2]{Fieseler} and \cite[Theorem 6.2]{DL}. 
We also note 
 that $R_{\tau,\sigma}(T)=T^{\dim (\tau )-n}$ whenever $\dim(\sigma)-\dim(\tau)\leq 2$, where $n=\dim(Y)$. 
Indeed, in this case $V:=V(\tau)\cap U_{\sigma}$ is a simplicial toric variety, hence $IC_V=\rat_V[\dim(V)]$. In particular, since
$R_{\tau, \tau}(T) =T^{\dim (\tau)-n}$ is invertible, it follows that the function $R$ has an inverse 
$$\widetilde{R} \colon \{(\tau,\sigma)\in\Delta_Y\times\Delta_Y\mid \tau\subseteq\sigma\} \to \zed [T,T^{-1}]$$ 
with respect to the convolution 
on the incidence algebra corresponding to the poset $\Delta_Y$. We set $\widetilde{r}_{\tau,\sigma}=\widetilde{R}_{\tau,\sigma}(1).$
\end{rmk}

The function $\widetilde{R}$ will feature in the description of the supports of a toric fibration.
 
 \begin{rmk}\label{rmk_combinatorics}
It follows from \cite[Proposition 8.1]{Stanley3} that, up to signs and powers of $T$, the function $\widetilde{R}$ is just 
the function $R$ associated with the dual poset. We are grateful to T. Braden for pointing this out to us. \end{rmk}

 Given a toric fibration $f\colon X \to Y$, we define the functions
 $$
 P_f\colon \Delta_Y \to   \zed [T,T^{-1}] \quad\text{and}\quad p_f\colon \Delta_Y \to \zed
 $$
 by
 $$P_{f, \sigma}(T):=\sum_{k\in\zed} \dim_{\rat}H^k(f^{-1}(x_{\sigma}), IC_X)T^k   \mbox{ and } p_{f, \sigma }:=P_{f, \sigma}(1)=\dim_{\rat}H^*(f^{-1}(x_{\sigma}), IC_X).$$
 It is not a priori clear that $ P_{f, \sigma}(T)$   and $p_{f,\sigma}$ are combinatorially determined, but this 
 follows from Theorem~\ref{thm_p_sigma} below. 
 Finally, we define 
 $$S\colon \Delta_Y \to   \zed [T,T^{-1}] \quad\text{as}\quad
 S_{\tau}(T)=\sum_{b\in\zed} s_{\tau,b}T^b,$$
 where the $s_{\tau,b}$ are the multiplicities defined in  Theorem~\ref{toric_dec_thm}.
 It follows from  i) in Theorem~\ref{toric_dec_thm} that $S_{\tau}(T)=S_{\tau}(T^{-1})$.
 Using the invariants $P_{f, \sigma}$ and $\widetilde{R}_{\tau,\sigma}$, we can now describe the supports 
 of any toric fibration.

\begin{tm}\label{form_general}
Suppose that we are in the setting of Theorem~\ref{toric_dec_thm}.
With the above definitions, for every $\tau\in\Delta_Y$, we have
$$S_\tau(T)= \sum_{\sigma \subseteq \tau }\widetilde{R}_{\sigma, \tau}(T)P_{f,\sigma} (T).$$
\end{tm}

\begin{proof}
We proceed as in the proof of Proposition~\ref{form_both_simplicial}. Consider the decomposition given by Theorem~\ref{toric_dec_thm}:
$$Rf_*(IC_X)\simeq\bigoplus_{\tau\in\Delta_Y}\bigoplus_{b\in\zed}IC_{V(\tau)}^{\oplus s_{\tau,b}}[-b].$$
Let $\sigma\in\Delta_Y$.
By taking the stalk at $x_{\sigma}$ and computing the $i^{\rm th}$ cohomology, we obtain
$$\dim_{\rat}H^i(f^{-1}(x_{\sigma}),IC_X)=\sum_{\tau\subseteq\sigma}\sum_{b\in\zed}s_{\tau,b}\cdot\dim_{\rat}{\mathcal H}^{i-b}(IC_{V(\tau)})_{x_{\sigma}},$$
which gives the equality  
$$
P_{f, \sigma}(T)=\sum_{\tau\subseteq\sigma} R_{\tau,\sigma}(T)S_{\tau}(T)
$$
in $\zed[T,T^{-1}]$.
Since $\widetilde{R}$ is the inverse of $R$ with respect to convolution, we conclude
$$
S_{\tau}(T)=\sum_{\sigma\subseteq\tau}\widetilde{R}_{\sigma,\tau}(T)\cdot P_{f, \sigma}(T).
$$
This completes the proof.
\end{proof}

Evaluating at  $T=1$ we find the following useful criterion for a stratum to be a support of the map $f$.
\begin{cor}\label{useful-formula}
In the setting of Theorem~\ref{toric_dec_thm}, we have
$$\delta_{\tau}=\sum_{\sigma\subseteq\tau}\widetilde{r}_{\sigma,\tau}\cdot p_{f, \sigma}.$$
\end{cor}


\begin{cor}\label{form_general_simpl_source}
If we are in the setting of Theorem~\ref{toric_dec_thm} and $X$ is simplicial with $\dim(X)=d_X$, then for every $\tau\in\Delta_Y$, we have
$$
S_{\tau}(T)=\sum_{\sigma\subseteq\tau}\widetilde{R}_{\sigma,\tau}(T)\cdot \left( \sum_{m\geq 0} \left(\sum_{\ell\geq m}(-1)^{\ell-m}{{\ell}\choose m}d_{\ell}(X/\sigma)\right)T^{2m-d_X}\right)
$$
and 
$$
\delta_{\tau}=\sum_{\sigma\subseteq\tau}\widetilde{r}_{\sigma,\tau}\cdot d_0(X/\sigma).
$$
\end{cor}

\begin{proof}
Since $X$ is simplicial, we have $IC_X=\rat_X[d_X]$. The first equality follows immediately from Theorem \ref{form_general} and Corollary \ref{Betti_fib}. The second equality is then obtained evaluating at $T=1$:
$$p_{f, \sigma}=\sum_{i\in\zed}\dim_{\rat}H^{i+d_X}(f^{-1}(x_{\sigma}),\rat)=\chi(f^{-1}(x_{\sigma}))=d_0(X/\sigma).$$
\end{proof}

By Remark~\ref{rmk_combinatorial_invariant}, in order to show that the formulas for the invariants $S_\tau$ and $\delta_{\tau}$ in Theorem~\ref{form_general} and Corollary~\ref{useful-formula}
 only depend on combinatorics,
it is enough to show that the invariants $P_{f, \sigma}$ only depend on combinatorics. This is implied by the following theorem. When $Y$ is a point and $X$ is projective,
this is a consequence of the results in \cite{Fieseler}.

\begin{tm}\label{thm_p_sigma}
If $f\colon X\to Y$ is a toric fibration, then for every $\sigma\in\Delta_Y$, we have
\begin{equation}\label{eq_thm_p_sigma}
P_{f, \sigma }(T)=\sum_{\tau} R_{0, \tau}(T)(T^2-1)^{{\rm codim}(\tau)-{\rm codim}(\sigma)} \quad
\text{and}\quad 
p_{f, \sigma}=\sum_{\tau}r_{0,\tau},
\end{equation}
where in the first formula the sum is over all the cones  $\tau$ in $\Delta_X$ with $f_*(\tau)=\sigma$, while in the second formula the $\tau$ are only those which, in addition, satisfy ${\rm codim}(\tau)={\rm codim}(\sigma)$. 
\end{tm}


\begin{proof}
The second statement follows from the first by evaluating it at $T=1$, so it is enough to prove the first statement.

First we prove that $H^i(f^{-1}(x_{\sigma}),IC_X) $ is pure, so that it is enough to determine its Hodge-Deligne polynomial.
After replacing $Y$ by $U_{\sigma}$, we may assume that $Y=U_{\sigma}$. 
Moreover, it is easy to see using Lemma~\ref{lm_prod_str} that we may assume 
that $x_\sigma$  is a fixed point. In this case, it follows from Lemma~\ref{retraction_lemma} (see also Remark~\ref{rmk2_retraction_lemma}) that
$$H^i(f^{-1}(x_{\sigma}),IC_X)  =H^k(X,IC_X),$$ 
therefore, as discussed in Remark~\ref{mhs_intcoh},  $H^i(f^{-1}(x_{\sigma}),IC_X) $ is pure.

We proceed as in the proof of Corollary~\ref{Betti_fib}. Set $n= \dim X$.
By Lemma~\ref{Omega_constructibility}, the restriction of the complex $IC_X$  to every torus orbit $O(\tau)$ in $f^{-1}(y)$ 
is a complex with {\em constant cohomology sheaves} ${\mathcal H}^{k}(IC_{X})_{x_\tau}$, underlying a pure Hodge-Tate structure of weight $k+n$.

Set  $t=\dim O(\tau)$.
Since $H^p_c(O(\tau))\cong \rat(p-t)^{\oplus { t\choose{p}}}$, the compact cohomology group 
$$
H^p_c(O(\tau), {\mathcal H}^{q}(IC_{X})_{x_\tau} )
$$
has a Hodge structure of Hodge-Tate type and weight $2(p-t)+q+n$
As it is well-known, see \ci[p. 571]{dM} or \ci[Example~5.2(1)]{CLMS}, the differentials in the hypercohomology spectral sequence
$$
E_2^{p \, q}=H^p_c(O(\tau), {\mathcal H}^{q}(IC_{X})_{x_\tau} ) \Rightarrow H^{p+q}_c(O(\tau), IC_{X} )
$$
are compatible with the Hodge structure , and they are therefore forced to vanish.
It follows that the Hodge-Deligne polynomial of $H^{p+q}_c(O(\tau), IC_{X} )$ is  $R_{0, \tau}(T)(uv-1)^t$. 
Adding over all torus orbits contained in $f^{-1}(x_{\sigma})$ and setting $u=v=T$, we obtain the statement.

\end{proof}

\begin{rmk}\label{rmk_KS3}
Due to the combinatorial nature of the definition of the local $h$-polynomial in \cite{KS}, the proofs of the analogues of Theorems~\ref{form_general} and \ref{thm_p_sigma}
in that setting are more elementary. One can then use the results of this section to write our invariants $s_{\tau,b}$ as coefficients of local $h$-polynomials. 
\end{rmk}

\providecommand{\bysame}{\leavevmode \hbox \o3em
{\hrulefill}\thinspace}


\begin{thebibliography}{Ful93}

\bibitem[BBD]{BBD}
A.~A.~Beilinson, J.~Bernstein, and P.~Deligne, Faisceaux pervers, in \emph{Analysis and topology on singular spaces, I (Luminy, 1981)}, 5--171, Ast\'{e}risque, 100, Soc. Math. France, Paris, 1982.

\bibitem[BL]{BL}
J.~Bernstein, V.~Lunts,  \emph{Equivariant sheaves and functors}, Lecture Notes in Mathematics, 1578, Springer-Verlag, Berlin, 1994.



\bibitem[CLS]{CLS}
D.~Cox, J.~Little, and H.~Schenck,  \emph{Toric varieties}, Graduate Studies in Mathematics, 124, American Mathematical Society, Providence, RI, 2011.


\bibitem[CMS]{CMS}S.E.~Cappell, L.G~ Maxim, J.L.~Shaneson,
Hodge genera of algebraic varieties I,
Comm. on Pure and Appl. Math. 61 (3), 422--449.

\bibitem[CLMS]{CLMS}S.E.~Cappell, A.~Libgober, L.~Maxim, J.L.~Shaneson, Julius,  Hodge genera of algebraic varieties II,
Math. Ann. 345 (2009), 925--972.



\bibitem[deC1]{decjag}
M. A. de Cataldo, 
The perverse filtration and the Lefschetz Hyperplane Theorem, II, 
J. Algebraic Geom. 21 (2012), 305--345.

\bibitem[deC2]{deC2}M.A.~de Cataldo, 
Proper toric maps over finite fields, Internat. Math. Res. Notices, to appear.


\bibitem[dM]{dM} M.A.~de Cataldo, L.~Migliorini, The decomposition theorem, perverse sheaves and the topology of algebraic maps. Bull. Amer. Math. Soc. (N.S.) 46 (2009), 535Ð-633.

\bibitem[Del]{Deligne}
P.~Deligne, Th\'{e}orie de Hodge, III,  Inst. Hautes \'{E}tudes Sci. Publ. Math. No. 44 (1974), 5--77.

\bibitem[DL]{DL}
J.~Denef and F.~Loeser, Weights of exponential sums, intersection cohomology, and Newton polyhedra, Invent. Math. 106 (1991), 275--294. 

\bibitem[Fie]{Fieseler}
K.-H.~Fieseler, Rational intersection cohomology of projective toric varieties, J. Reine Angew. Math. 413 (1991), 88--98. 

\bibitem[Ful]{Ful93}
W.~Fulton, \emph{Introduction to toric varieties}, Ann. of Math.
Stud. \textbf{131}, The William H. Rover Lectures in Geometry,
Princeton Univ. Press, Princeton, NJ, 1993.

\bibitem[Ive]{iversen}
B. Iversen, {\em Cohomology of Sheaves}, Universitext, 
Springer-Verlag, Berlin Heidelberg, 1986.

\bibitem[Jur]{Jur77}
 J.~Jurkiewicz, An example of algebraic torus action which determines the nonfiltrable decomposition, Bull. Acad. Polon. Sci. S\'{e}r. Sci. Math. Astronom. Phys. 25 (1977), 1089--1092.

\bibitem[KS]{KS}E.~Katz, A.~Stapledon, Local $h$-polynomials, invariants of subdivisions, and mixed Ehrhart theory, Adv. Math. 
286, 2 (2016), 181Ð239.

\bibitem[Kir]{Kirwan}
F.~Kirwan, Intersection homology and torus actions, J. Amer. Math. Soc. 1 (1988), 385--400.

\bibitem[Mus]{Mustata}
M.~Musta\c{t}\u{a}, Lecture notes on toric varieties, available at $\,\,\,\,\,\,\,\,\,\,\,\,\,\,\,\,\,\,\,\,\,\,\,\,\,\,\,\,\,\,\,\,\,\,\,\,\,\,\,\,\,\,\,\,\,\,\,\,\,\,\,\,\,\,\,\,\,\,\,\,\,\,\,\,\,\,\,\,\,\,\,\,\,\,\,\,\,\,\,\,\,\,$
\emph{http://www-personal.umich.edu/\,\,$\tilde{}$mmustata/toric$\_$var.html}.

\bibitem[Rota]{Rota64}
G.~Rota, On the foundations of combinatorial theory, I, Theory of M\"{o}bius functions, Z. Wahrscheinlichkeitstheorie und Verw. Gebiete 2 (1964), 340--368.

\bibitem[Sai]{Saito}
M.~Saito, Mixed Hodge modules, Publ. Res. Inst. Math. Sci. 26 (1990), 221--333. 

\bibitem[Sum]{Sumihiro}
H.~Sumihiro, Equivariant completion. J. Math. Kyoto Univ. 14 (1974), 1--28.

\bibitem[Sta1]{Stanley1}
R.~Stanley, The number of faces of a simplicial convex polytope,
Adv. in Math. 35 (1980), 236--238. 

\bibitem[Sta2]{Stanley2}
R.~Stanley, Generalized $H$-vectors, intersection cohomology of toric varieties, and related results, in  \emph{Commutative algebra and combinatorics (Kyoto, 1985)}, 187--213, Adv. Stud. Pure Math., 11, North-Holland, Amsterdam, 1987.

\bibitem[Sta3]{Stanley3}
R.~Stanley, Subdivisions and local h-vectors.
J. Amer. Math. Soc. 5 (1992), 805--851.

\bibitem[Sta4]{Stanley4}
R.~Stanley, {\em Enumerative combinatorics I,} Cambridge Studies in Advanced Mathematics 49, Cambridge University Press, 1997.

\end{thebibliography}
\end{document}